\documentclass[11pt]{article}
    \usepackage{url}
    \usepackage{verbatim}
    \usepackage[titletoc]{appendix}
    \usepackage{graphicx}
    \usepackage{nicefrac}
    \usepackage{longtable}
    \usepackage{color}
 \usepackage{epic,eepic,euscript,verbatim,afterpage}
\usepackage{float,bm,paralist,color}
\usepackage{hyperref}
\usepackage{amsmath,amsthm}
\usepackage{amssymb}
\usepackage{amsfonts}    
    \usepackage{fancyhdr,graphics,graphicx}
\usepackage{epsf,epsfig}
\usepackage{multirow}
\usepackage{diagbox}
\usepackage{cancel}
\usepackage{overpic}
\usepackage{subfig,epstopdf}
\usepackage{array}
\usepackage{cite}
    \hypersetup{hidelinks}

\newtheorem{rem}{Remark}[section]
\newtheorem{thm}{Theorem}[section]
\newtheorem{cor}{Corollary}[section]
\newtheorem{lm}{Lemma}[section]
\newtheorem{defi}{Definition}[section]
\newtheorem{prop}{Proposition}[section]

\newcommand{\nabh}{\nabla_{\! h}}

\newcommand{\Dh}{\Delta_{\! h}}
\newcommand{\bu}{\mathbf{u}}
\newcommand{\bU}{\mathbf{U}}
\newcommand{\bv}{\mathbf{v}}

\newcommand{\B}[3]{\mathcal{B}\left( #1, #2, #3\right)}

\newcommand{\Tri}

\newcommand{\eu}{{\bf e}_{\textbf{u}}}
\newcommand{\teu}{\tilde{\bf e}_{\textbf{u}}^{n+\hf }}
\newcommand{\beu}{\bar{\bf e}_{\textbf{u}}^{n+\hf }}
\newcommand{\bhu}{\hat{\mbox{\textbf{u}}}}

\newcommand{\ehu}{{\hat{\textbf{e}}}_{\textbf{u}}}

\newcommand{\behu}{\bar{\hat{\textbf{e}} }_{\textbf{u}}^{n+\hf }}

\newcommand{\emu}{e_{\mu}^{n+\hf }}

\newcommand{\bep}{\bar{e}_{p}^{n+\hf }}

\newcommand{\ephi}{e_{\phi}}
\newcommand{\tephi}{\tilde{e}_{\phi}^{n+\hf }}

\newcommand{\bbephi}{\bar{\bar{e}}_{\phi}^{n+\hf }}

\newcommand{\muhalf}{\mu^{n+\hf }}
\newcommand{\Muhalf}{\mbox{M}^{n+\hf }}
\newcommand{\ext}[1]{\tilde{#1}^{n+\hf }}
\newcommand{\imp}[1]{\bar{#1}^{n+\hf }}
\newcommand{\impp}[1]{\bar{\bar{#1}}^{n+\hf }}

\newcommand{\mAh}{\mathcal{A}_h}

\newcommand{\hf}{\frac{1}{2}}
\newcommand{\nrm}[1]{\left\| #1 \right\|}

\newcommand{\nrminf}[1]{\left\| #1 \right\|_{\infty}}

\newcommand\dt {{\Delta t}}

\newcommand{\eipC}[2]{\left< #1 , #2 \right>_{c}}
\newcommand{\eipew}[2]{\left< #1 , #2 \right>_{ew}}
\newcommand{\eipns}[2]{\left< #1 , #2 \right>_{ns}}
\newcommand{\eipvec}[2]{\left< #1 , #2 \right>}

     \def\n{\mbox{\boldmath $n$}} 
     \def\u{\mbox{\boldmath $u$}}    
                
     \def\v{\mbox{\boldmath $v$}}

\graphicspath{{Figures/}}
	\title{Convergence analysis of a second order numerical scheme for the Flory-Huggins-Cahn-Hilliard-Navier-Stokes system}

	\date{\today}

\begin{document}
	
\author{
Wenbin Chen\thanks{
Shanghai Key Laboratory for Contemporary Applied Mathematics, School of Mathematical Sciences; Fudan University, Shanghai, China 200433 ({\tt wbchen@fudan.edu.cn})} 
\and 
Jianyu Jing\thanks{School of Mathematical Sciences; Fudan University, Shanghai, China 200433 ({\tt jyjing20@fudan.edu.cn})} 
\and 
Qianqian Liu\thanks{School of Mathematical Sciences; Fudan University, Shanghai, China 200433 ({\tt qianqianliu21@m.fudan.edu.cn})} 
\and
Cheng Wang\thanks{Mathematics Department; University of Massachusetts; North Dartmouth, MA 02747 ({\tt corresponding author: cwang1@umassd.edu})}
\and
Xiaoming Wang\thanks{Department of Mathematics and Statistics, Missouri University of Science and Technology, Rolla, MO 65409, USA
	({\tt xiaomingwang@mst.edu})}
			}

 	\maketitle
	\numberwithin{equation}{section}
	
\begin{abstract}	 

We present an optimal rate convergence analysis for a second order accurate in time, fully discrete finite difference scheme for the Cahn-Hilliard-Navier-Stokes (CHNS) system, combined with logarithmic Flory-Huggins energy potential. The numerical scheme has been recently proposed, and the positivity-preserving property of the logarithmic arguments, as well as the total energy stability, have been theoretically justified. In this paper, we rigorously prove second order convergence of the proposed numerical scheme, in both time and space. Since the CHNS is a coupled system, the standard $\ell^\infty (0, T; \ell^2) \cap \ell^2 (0, T; H_h^2)$  error estimate could not be easily derived, due to the lack of regularity to control the numerical error associated with the coupled terms. Instead, the $\ell^\infty (0, T; H_h^1) \cap \ell^2 (0, T; H_h^3)$ error analysis for the phase variable and the $\ell^\infty (0, T; \ell^2)$ analysis for the velocity vector, which shares the same regularity as the energy estimate, is more suitable to pass through the nonlinear analysis for the error terms associated with the coupled physical process. Furthermore, the highly nonlinear and singular nature of the logarithmic error terms makes the convergence analysis even more challenging, since a uniform distance between the numerical solution and the singular limit values of is needed for the associated error estimate. Many highly non-standard estimates, such as a higher order asymptotic expansion of the numerical solution (up to the third order accuracy in time and fourth order in space), combined with a rough error estimate (to establish the maximum norm bound for the phase variable), as well as a refined error estimate, have to be carried out to conclude the desired convergence result. To our knowledge, it will be the first work to establish an optimal rate convergence estimate for the Cahn-Hilliard-Navier-Stokes system with a singular energy potential.

	\bigskip

\noindent
{\bf Key words and phrases}:
Cahn-Hilliard-Navier-Stokes system, Flory-Huggins energy potential, Crank-Nicolson approximation, optimal rate convergence analysis, higher order asymptotic expansion, rough and refined error estimates


\noindent
{\bf AMS subject classification}: \, 35K35, 35K55, 49J40, 65M06, 65M12	
\end{abstract}
	





\section{Introduction}

A bounded domain $\Omega \subset \mathbb{R}^d$ ($d=2$ or $d=3$) is considered. For simplicity, it is assumed that $\Omega= (0,1)^2$, 
and an extension to the three-dimensional (3-D) domain 
would be straightforward. 

In the phase field formulation, a point-wise bound, $-1 < \phi < 1$, is assumed for the phase variable $\phi$. For any $\phi \in H^1 (\Omega)$ with this bound, the Flory-Huggins free energy is given by   
	\begin{equation}
	\label{CH energy}
E(\phi)=\int_{\Omega}\left( ( 1+ \phi) \ln (1+\phi) + (1-\phi) \ln (1-\phi) - \frac{\theta_0}{2} \phi^2 +\frac{\epsilon^2}{2}|\nabla \phi|^2\right) d {\bf x} ,
	\end{equation}
in which $\epsilon>0$, $\theta_0>0$ are certain physical parameter constants associated with the diffuse interface width and inverse temperature, respectively; see the related references~\cite{cahn1996, elliott92a, doi13, elliott96b}, etc. 

In addition to the phase field evolution, the fluid motion has to be considered in the physical process. In particular, the dynamical equations of the Cahn-Hilliard-Navier-Stokes (CHNS) system~\cite{liu03} are formulated as   
\begin{align} 
  & 
   \bu_t + \bu \cdot \nabla \bu + \nabla p  - \nu \Delta \bu = - \gamma \phi \nabla \mu  , 
   \label{equation-CHNS-1} 
\\
 & 
 \phi_t + \nabla \cdot ( \phi \bu) = \Delta \mu ,  \label{equation-CHNS-2}   
\\
  & \mu := \delta_\phi E 
 = \ln ( 1 + \phi) - \ln ( 1- \phi) 
  - \theta_0 \phi - \epsilon^2  \Delta \phi  ,  \label{equation-CHNS-3}  
\\
  & \nabla \cdot \bu = 0 , \label{equation-CHNS-4}    
\end{align}  
with no-flux and no-penetration free-slip boundary conditions: 
	\begin{equation} 
\partial_{n} \phi=\partial_{n} \mu=0, \quad \bu \cdot \n = \partial_n (\bu\cdot \boldsymbol{\tau}) = 0  , \quad\quad\mbox{on} \ \partial \Omega \times(0, T]. \label{BC-1} 
	\end{equation} 
In this system, $\bu$ is the advective velocity, $p$ is the pressure variable, and $\nu>0$ is the kinematic viscosity. The constant $\gamma > 0$ is associated with surface tension, and term $-\gamma\phi\nabla\mu$ corresponds to a diffuse interface approximation of the singular surface force. 
Moreover, the following energy dissipation law could be carefully derived for this coupled physical system: 
\begin{equation} 
  E'_{total} (t) 
  =   - \int_\Omega |  \nabla \mu |^2 d {\bf x}  
  - \frac{\nu}{\gamma}  \int_\Omega |  \nabla \bu |^2 d {\bf x}   \le 0 ,   \quad 
  E_{total} = E (\phi) + \frac{1}{2 \gamma} \| \bu \|^2  .  
    \label{total energy-dissipation-1} 
\end{equation}   
See the related PDE analysis works of various phase-field-fluid coupled system~\cite{abels09b, liu03, lowengrub98}, etc. 
  

An efficient and energy stable numerical approximation to the Cahn-Hilliard-Navier-Stokes (CHNS) system has always been an attractive and challenging issue, due to the highly coupled nature between the phase field evolution and fluid motion. Many linear, decoupled and energy stable numerical schemes have been applied to the CHNS system~\cite{feng06, han17, kay07, kim03, shen10b, shen2014, shen2015, Yang2021, yang17d, Zhao2021b, Zhao2021a}, with a polynomial approximation to the energy potential in the phase field formulation. With such a polynomial approximation, the singularity issue of the energy functional has been avoided, and the linear and decoupled numerical solvers have demonstrated its advantages in terms of computational efficiency. However, an extension of this idea to  the Flory-Huggins-Cahn-Hilliard-Navier-Stokes (FHCHNS) system~\eqref{equation-CHNS-1}-\eqref{equation-CHNS-4} will face a serious difficulty, which comes from the singularity of the logarithmic term in the Flory-Huggins energy formulation. For example, the numerical solutions created by either the linear stabilization method~\cite{shen10b}, the invariant energy quadratization (IEQ) method~\cite{yang17d}, or scalar auxiliary variable (SAV) method~\cite{Yang2021}, may not preserve the positivity of the phase variable at the next time step, so that the energy could even not be defined after a single-step computation. Of course, many remedy efforts may be made, such as an extension of the energy functional definition even if the phase variable does not preserve the positivity. On the other hand, these efforts may introduce non-physical solutions in the long-time simulation, so that a theoretical justification of both the positivity-preserving property and energy stability for the FHCHNS system has always been highly desirable. To accomplish these theoretical properties, an implicit treatment for the nonlinear and singular terms turns out to be necessary. In fact, for various Cahn-Hilliard-Fluid physical systems, with a polynomial approximation in the energy potential expansion, such an implicit numerical approach, so called the convex splitting method, has been widely used, while both the energy stability and optimal rate convergence analysis have been extensively reported~\cite{chen19a, chen22c, chen16, diegel17, feng12, han15, liuY17}. Meanwhile, most existing numerical works for the Cahn-Hilliard-Fluid system have focused on the polynomial approximation version, and a theoretical numerical analysis of the FHCHNS system, with a logarithmic energy potential, has been very limited. A pioneering work~\cite{chen22b} proposes a first order accurate (in time) numerical scheme for the FHCHNS system~\eqref{equation-CHNS-1}-\eqref{equation-CHNS-4}, in which the convex splitting approach is employed to the chemical potential in the phase field part, and semi-implicit discretization is applied to the fluid convection and coupled terms in the physical system. Both the positivity-preserving and total energy stability properties have been proved, and this work provides a theoretical analysis for the FHCHNS system~\eqref{equation-CHNS-1}-\eqref{equation-CHNS-4}, for the first time in the existing literature. 

Moreover, a second order (in time) numerical approximation to the (FHCHNS) system~\eqref{equation-CHNS-1}-\eqref{equation-CHNS-4} turns out to be highly non-trivial, since a second order numerical design for the nonlinear logarithmic terms would be very challenging to preserve both the positivity-preserving property and the total energy stability. Such a second order accurate, finite difference scheme has been proposed in a recent article~\cite{chen24a}, with both of these theoretical properties rigorously established. In more details, a modified Crank-Nicolson approximation is applied to the singular logarithmic nonlinear term, while the expansive term is updated by an explicit second order Adams-Bashforth extrapolation, and an alternate temporal stencil is used for the surface diffusion term. Furthermore, a nonlinear artificial regularization term is added in the chemical potential approximation, and this term ensures the positivity-preserving property for the logarithmic arguments. The convective term in the phase field evolutionary equation is updated in a semi-implicit way, with second order accurate temporal approximation. The fluid momentum equation could also be computed by a semi-implicit algorithm. The resulting numerical system is proven to be uniquely solvable, positivity-preserving and unconditionally stable in terms of total energy. In fact, an iteration process is constructed to establish these theoretical properties. 

A few interesting numerical simulation results have been presented. On the other hand, the convergence analysis for the FHCHNS system~\eqref{equation-CHNS-1}-\eqref{equation-CHNS-4} remained an open problem, even for the first order accurate scheme. In this article, we provide an optimal rate convergence analysis for the fully discrete second order scheme formulated in~\cite{chen24a}, which is shown to be second order accurate in both time and space. Because of the highly coupled nature of the CHNS system, the standard $\ell^\infty (0, T; \ell^2) \cap \ell^2 (0, T; H_h^2)$ error estimate could not pass through, which comes from the lack of regularity to control the error inner products associated with the nonlinear coupled terms in the phase evolutionary equation and the momentum equation. To overcome this difficulty, the $\ell^\infty (0, T; H_h^1) \cap \ell^2 (0, T; H_h^3)$ error analysis, which shares the same regularity as the energy estimate, would be appropriate to pass through the associated nonlinear analysis; see the related reference works~\cite{chen19a, chen22c, chen16, diegel17, GuoY2024a, liuY17}, etc. Meanwhile, all these existing error estimate works for the phase field-fluid coupled system have been associated with the polynomial approximation in the free energy expansion. In comparison, for the CHNS system~\eqref{equation-CHNS-1}-\eqref{equation-CHNS-4} with Flory-Huggins energy potential, the highly nonlinear and singular nature of the logarithmic error terms makes the convergence analysis even more challenging. In particular, a uniform distance between the numerical solution and the singular limit values of $\pm 1$ is needed to pass through the associated error estimate. In turn, many highly non-standard techniques have to be involved in the theoretical analysis. First, a higher order asymptotic expansion, up to third order accuracy in time and fourth order accuracy in space, has to be performed with a careful linearization technique. Such a higher order asymptotic expansion enables one to obtain a rough error estimate, so that to the maximum norm bound for the phase variable could be derived. As a direct consequence, this maximum norm bound ensures a uniform distance between the numerical solution and the singular limit values, which will play a crucial role in the subsequent analysis. Finally, a refined error estimate is carried out to accomplish the desired convergence result. To our knowledge, this will be the first work to provide an optimal rate convergence estimate for the Cahn-Hilliard-Navier-Stokes system with singular energy potential. 

The rest of the article is organized as follows. In Section~\ref{sec:numerical scheme}, we review the fully discrete finite difference scheme and state the main theoretical result. The optimal rate convergence analysis and error estimate are presented in Section~\ref{sec:convergence}. Finally, some concluding remarks are made in Section~\ref{sec:conclusion}.


	\section{Numerical scheme}  \label{sec:numerical scheme} 	
	
		\subsection{The finite difference spatial discretization}
For simplicity, we only consider the two dimensional domain $\Omega=(0,1)^2$. The three dimensional case can be similarly extended. In this domain, we denote the uniform spatial grid size $h=\frac{1}{N}$, with $N$ a positive integer. To facilitate the theoretical analysis, the marker and cell (MAC) grid \cite{Harlow1965} is used: the phase variable $\phi$, the chemical potential $\mu$ and the pressure field $p$ are defined on the cell-centered mesh points $\left(\left(i+\hf \right)h,\,\left(j+\hf h\right)\right),\ 0\leq i,\ j \le N$; for the velocity field $\bu = (u^x,\,u^y)$, the $x$-component of the velocity will be defined at the east-west cell edge points $\left(ih,\,\left(j+\hf h\right)\right),\ 0\leq i\leq N+1,\ 0\leq j \le N$, 
while the $y$-component of the velocity is located at the north-south cell edge points $\left(\left(i+\hf \right)h,\,jh\right)$. 

For a function $f(x,y)$, the notation $f_{i+\hf ,\, j+\hf }$ represents the value of $f\left(\left(i+\hf \right)h,\, \left(j+\hf \right)h\right)$. Of course, $f_{i+\hf ,\, j}$, $f_{i,\, j+\hf }$ could be similarly introduced. In turn, the following difference operators are introduced:
\begin{align}
	&(D^c_xf)_{i,\, j+\hf } = \frac{f_{i+\hf ,\, j+\hf }-f_{i-\hf ,\, j+\hf }}{h},\qquad
	(D^c_yf)_{i+\hf ,\, j} = \frac{f_{i+\hf ,\, j+\hf }-f_{i+\hf ,\, j-\hf }}{h},	\label{center-diff-operator}\\
	&(D^{ew}_xf)_{i+\hf ,\, j+\hf } = \frac{f_{i+1,\, j+\hf }-f_{i,\, j+\hf }}{h},\qquad
	(D^{ew}_yf)_{i,\, j} = \frac{f_{i,\, j+\hf }-f_{i,\, j-\hf }}{h},	\label{ew-diff-operator}\\
	&(D^{ns}_xf)_{i,\, j} = \frac{f_{i+\hf ,\, j}-f_{i-\hf ,\, j}}{h},\qquad
	(D^{ns}_yf)_{i+\hf ,\, j+\hf } = \frac{f_{i+\hf ,\, j+1}-f_{i+\hf ,\, j}}{h}.	\label{ns-diff-operator}
\end{align}
The boundary formulas may vary with different boundary conditions. 
With homogeneous Neumann boundary condition, \eqref{center-diff-operator} becomes
\begin{equation} \label{boundary-formula-1} 
	(D^c_xf)_{0,\, j+\hf } = (D^c_xf)_{N,\, j+\hf } = (D^c_yf)_{i+\hf ,\, 0} = (D^c_yf)_{i+\hf ,\, N} = 0.
\end{equation} 
The associated formulas for \eqref{ew-diff-operator}-\eqref{ns-diff-operator} could be analogously derived. 

In turn, with a careful evaluation of boundary differentiation formula~\eqref{boundary-formula-1}, the discrete boundary condition associated with cell-centered function is given by the following definition, in which the ``ghost" points are involved. The boundary formulas for the edge-centered function could be similarly derived.

	\begin{defi} \label{defi: BC} 
A cell-centered function $\phi$ is said to satisfy homogeneous Neumann boundary condition, and we write $\boldsymbol{n} \cdot \nabla_{h} \phi=0$, iff $\phi$ satisfies
	\begin{equation*} 
 \phi_{-\frac12,j+\frac12}=\phi_{\frac12,j+\frac12}, \, \, \, \phi_{N+\frac12,j+\frac12}=\phi_{N-\frac12,j+\frac12} , \, \,  
\phi_{i+\frac12,-\frac12}=\phi_{i+\frac12,\frac12} , \, \, \,  \phi_{i+\frac12,N+\frac12}=\phi_{i+\frac12,N-\frac12} .  
	\end{equation*} 
A discrete function $\boldsymbol{f}=(f^{x}, f^{y})^{T}$, with two components evaluated at east-west and north-south mesh points, is said to satisfy no-penetration boundary condition, $\boldsymbol{n}\cdot\boldsymbol{f}=0$, iff we have
	\begin{equation*} 
 f_{0, j+\frac12}^x= f_{N, j+\frac12}^x=0, \, \, \,  
	\\
 f_{i+\frac12, 0}^{y} = f_{i+\frac12, N}^{y}=0 ,  
	\end{equation*} 
and it is said to satisfy free-slip boundary condition iff we have
	\begin{equation*} 
 f_{i, -\frac12}^{x}=f_{i, \frac12}^{x},  \, \, \, 
 f_{i, N+\frac12}^{x}=f_{i, N-\frac12}^{x}, \, \, \, 
 f_{-\frac12, j}^{y}=f_{\frac12, j}^{y} , \, \, \,  
 f_{N+\frac12, j}^{y}=f_{N-\frac12, j}^{y}. 
	\end{equation*} 
\end{defi}

In addition, the long stencil difference operator is also defined on the east-west cell edge points and north-south cell edge points:
\begin{align}\label{long-stencil-diff-operator}
	(\tilde{D}_xf)_{i,\, j+\hf } = \frac{f_{i+1,\, j+\hf }-f_{i-1,\, j+\hf }}{2h},\qquad
	(\tilde{D}_yf)_{i+\hf ,\, j} = \frac{f_{i+\hf ,\, j+1}-f_{i+\hf ,\, j-1}}{2h}.
\end{align}
With 
homogeneous Dirichlet boundary condition, \eqref{long-stencil-diff-operator} could be written as 
\begin{align}
	&(\tilde{D}_xf)_{0,\, j+\hf } = \frac{f_{1,\, j+\hf }-f_{-1,\, j+\hf }}{2h} = \frac{f_{1,\, j+\hf }}{h},\\
	&(\tilde{D}_xf)_{N,\, j+\hf } = \frac{f_{N+1,\, j+\hf }-f_{N-1,\, j+\hf }}{2h} = -\frac{f_{N-1,\, j+\hf }}{h},\\
	&(\tilde{D}_yf)_{i+\hf ,\, 0} = \frac{f_{i+\hf ,\, 1}-f_{i+\hf ,\, -1}}{2h} = \frac{f_{i+\hf ,\, 1}}{h},\\
	&(\tilde{D}_yf)_{i+\hf ,\, N} = \frac{f_{i+\hf ,\, N+1}-f_{i+\hf ,\, N-1}}{2h} = -\frac{f_{i+\hf ,\, N-1}}{h}.
\end{align}
For a grid function $f$, the discrete gradient operator is defined as 
\begin{align}
	&\nabh f = \left( (D^{\ell}_xf ),\  (D^{\ell}_yf )\right)^T,
\end{align}
where $\ell = c,\ ew,\ ns$ may depend on the choice of $f$. The discrete divergence operator of a vector gird function $\bu$, defined on the cell-centered points, turns out to be  
\begin{equation}
	\left(\nabh\cdot \bu\right)_{i+\hf ,\, j+\hf } = 
		(D^{ew}_xu^x)_{i+\hf ,\, j+\hf } + (D^{ns}_yu^y)_{i+\hf ,\, j+\hf }.
\end{equation}
The five point standard Laplacian operator is straightforward: 
\begin{equation}
	(\Dh f)_{r,\, s} = \frac{f_{r+1,\,s}+f_{r-1,\,s}+f_{r,\,s+1}+f_{r,\,s-1}-4f_{r,\,s}}{h^2},
\end{equation}
where $(r,\,s)$ may refer to $(i+\hf ,\,j+\hf )$, $(i+\hf ,\,j)$ and $(i,\,j+\hf )$.

For $\bu = (u^x,\,u^y)^T$, $\bv = (v^x,\,v^y)^T$, located at the staggered mesh points respectively, and the cell centered variables $\phi$, $\mu$, the nonlinear terms are evaluated as follows~\cite{chen24a}: 
\begin{eqnarray} 
	&& 
	\bu \cdot \nabla_h \bv = \left( \begin{array}{c} 
		u^x_{i,\,j+\hf} \tilde{D}_x v^x_{i,\,j+\hf} + {\cal A}_{xy} u^y_{i,\,j+\hf} \tilde{D}_y v^x_{i,\,j+\hf} \\ 
		{\cal A}_{xy} u^x_{i+\hf,\,j} \tilde{D}_x v^y_{i+\hf,\,j} + u^y_{i,\,j+\hf} \tilde{D}_y v^y_{i+\hf,\,j}  
	\end{array} \right)  ,  \label{FD-u-1} 
	\\
	&& 
	\nabla_h \cdot (\bv \bu^T) =  \left( \begin{array}{c} 
		\tilde{D}_x (u^x v^x)_{i,\,j+\hf} + \tilde{D}_y ( {\cal A}_{xy} u^y v^x)_{i,\,j+\hf}  \\ 
		\tilde{D}_x ({\cal A}_{xy} u^x v^y)_{i+\hf,\,j} + \tilde{D}_y ( u^y v^y)_{i+\hf,\,j}  
	\end{array} \right)  ,   \label{FD-u-2} 
	\\
	&& 
	{\cal A}_h \phi \nabla_h \mu   =  \left( \begin{array}{c} 
		( D^c_x \mu \cdot {\cal A}_x \phi )_{i,\,j+\hf}   \\ 
		( D^c_y \mu \cdot {\cal A}_y \phi )_{i+\hf,\,j} 
	\end{array} \right)  ,   \label{FD-u-3} 
	\\
	&& 
	\nabla_h \cdot ({\cal A}_h \phi \bu) = 
	D^{ew}_x (u^x {\cal A}_x \phi)_{i+\hf,\,j+\hf} + D^{ns}_y (u^y {\cal A}_y \phi)_{i+\hf,\,j+\hf}  ,   
	\label{FD-phi-1} 
\end{eqnarray} 
where the averaging operators are given by 
\begin{eqnarray} 
	&& 
	{\cal A}_{xy} u^x_{i+\hf,\,j} = \frac14 \left(  u^x_{i,\,j-\hf} + u^x_{i,\,j+\hf} 
	+ u^x_{i+1,\,j-\hf} + u^x_{i+1,\,j+\hf}\right)   ,   \label{FD-ave-1}    
	\\
	&&    
	{\cal A}_x \phi_{i,\,j+\hf} = \frac12 \left(  \phi_{i-\hf,\,j+\hf} + \phi_{i+\hf,\,j+\hf} \right) .  \label{FD-ave-3}         
\end{eqnarray}     
A few other average terms, such as ${\cal A}_{xy} u^y_{i,\,j+\hf}$, ${\cal A}_y \phi_{i+\hf,\,j}$,  could be similarly defined.

In addition, the discrete inner product needs to be defined to facilitate the theoretical analysis. Let $f$, $g$ be two grid functions evaluated on the cell-center points, the discrete $\ell^2$ inner product is given by
\begin{equation}\label{center-l2-product}
	\eipC{f}{g} = h^2\sum_{i=1}^{N}\sum_{j=1}^{N}f_{i+\hf,\,j+\hf}g_{i+\hf,j+\hf}.
\end{equation}
If $f$, $g$ are evaluated on the east-west points, \eqref{center-l2-product} becomes:
\begin{equation}\label{eastwest-l2-product}
	\eipew{f}{g} = h^2\sum_{i=1}^{N}\sum_{j=1}^{N}f_{i,\,j+\hf}g_{i,\,j+\hf}.
\end{equation}
If $f$, $g$ are evaluated on the north-south points, \eqref{center-l2-product} shifts into:
\begin{equation}\label{northsouth-l2-product}
	\eipns{f}{g} = h^2\sum_{i=1}^{N}\sum_{j=1}^{N}f_{i+\hf,\,j}g_{i+\hf,\,j}.
\end{equation}
Similarly, for two vector grid functions $\bu=\left(u^x,\,u^y\right)^T$, $\bv=\left(v^x,\,v^y\right)^T$ whose components are evaluated on east-west and north-south respectively, the vector inner product is defined as 
\begin{equation}
	\eipvec{\bu}{\bv}_1 = \eipew{u^x}{v^x} + \eipns{u^y}{v^y}.
\end{equation}
Consequently, the discrete $\ell^2$ norms, $\| \cdot \|_2$ can be naturally introduced. Furthermore, the discrete $\ell^p$, $1\leq p\leq\infty$ norms are needed in the nonlinear analysis. For $(r,\,s)=(i+\hf ,\,j+\hf )$, $(i+\hf ,\,j)$ or $(i,\,j+\hf )$, we introduce
\begin{equation}
	\nrminf{f}:=\max\limits_{r,\,s}\left|f_{r,\,s}\right|,\qquad
	\nrm{f}_p:= \Big( h^2\sum_{r=0}^{N}\sum_{s=0}^{N}\left|f_{r,\,s}\right|^p \Big)^{\frac{1}{p}},\quad1\leq p<\infty.
\end{equation} 

The discrete average is defined as $\overline{f} := \langle f , 1 \rangle_c$, for any cell centered function $f$. Moreover, an $\langle \cdot , \cdot \rangle_{-1,h}$ inner product and $\| \cdot \|_{-1, h}$ norm need to be introduced to facilitate the analysis in later sections. For any $\varphi  \in \mathring{\mathcal C}_{\Omega} := \left\{ f \middle| \ \langle f , 1 \rangle_c = 0 \right\}$, we define 
	\begin{equation} 
\langle \varphi_1, \varphi_2  \rangle_{-1,h} = \langle \varphi_1 ,  (-\Delta_h )^{-1} \varphi_2 \rangle_c , \quad \| \varphi  \|_{-1, h } = \sqrt{ \langle \varphi ,  ( - \Delta_h )^{-1} (\varphi) \rangle_c } ,
	\end{equation} 
where the operator $\Delta_h$ is paired with discrete homogeneous Neumann boundary condition.

The following summation by parts formula has been derived in~\cite{chen24a}. We recall these formulas, which will be useful in the convergence analysis. 

\begin{lm} \cite{chen24a} 
	For two discrete grid vector functions $\bu=\left(u^x,\,u^y\right)$, $\bv=(v^x,\,v^y)$, where $u^x$, $u^y$ and $v^x$, $v^y$ are defined on east-west and north-south respectively, and two cell centered functions $f$, $g$, the following identities are valid, if $\bu$, $\bv$, $f$, $g$ are equipped with periodic boundary condition, or $\bu$, $\bv$ are implemented with homogeneous Dirichlet boundary condition and homogeneous Neumann boundary condition is imposed for $f$ and $g$:
	\begin{align}
		&\eipvec{\bv}{\bu\cdot\nabh\bv}_1 +\eipvec{\bv}{\nabh\cdot (\bv\bu^T )}_1 =0, 
		\label{summation-1} \\
		&\eipvec{\bu}{\nabh f} _1= 0,\quad \mbox{if} \quad  \nabh\cdot \bu=0,  
		 \label{summation-2} \\
		-&\eipvec{\bv}{\Dh\bv}_1 =\nrm{\nabh\bv}_2^2, \label{summation-3} \\
		-&\eipC{f}{\Dh f} =\nrm{\nabh f}_2^2,  \label{summation-4} \\
		-&\eipC{g}{\nabh\cdot\left({\cal A}_hf\bu\right)}
		 =\eipvec{\bu}{{\cal A}_hf\nabh g}_1 .  \label{summation-5} 
	\end{align}
\end{lm}

The following Poincar\'{e}-type inequality will be useful in the later analysis. 
	\begin{prop} \label{prop: Poincare} 
(1) There are constants $C_0, \, C_1  >0$, independent of $h>0$, such that $\nrm{\phi}_2 \le C_0 \nrm{\nabla_h\phi}_2$, $\| \phi \|_{-1, h} \le C_1 \| \phi \|_2$, for all $\phi\in \mathring{\mathcal C}_{\Omega} := \left\{ f \middle| \ \langle f , 1 \rangle_c = 0 \right\}$. 

(2) For a velocity vector $\v$, with a discrete no-penetration boundary condition $\v \cdot \n =0$ on $\partial \Omega$, a similar Poincar\'e inequality is also valid: $\nrm{\v}_2 \le C_0 \nrm{\nabla_h\v}_2$, with $C_0$ only dependent on $\Omega$. 
	\end{prop}

		\subsection{The second order numerical scheme and the main theoretical results}

The second order in time accurate scheme has been proposed in~\cite{chen24a}. Given $\bu^n = \left((u^x)^n,\,(u^y)^n\right)$, $\bu^{n-1}=\left((u^x)^{n-1},\,(u^y)^{n-1}\right)$ evaluated at the MAC staggered grid, and $p^n$, $\phi^n$, $\phi^{n-1}$ located at the cell-centered grid, with $\nrminf{\phi^n}$, $\nrminf{\phi^{n-1}}<1$, we aim to find $\bhu^{n+1}$, $\bu^{n+1}$, $p^{n+1}$, $\phi^{n+1}$ that satisfy
\begin{align}
	&\frac{\bhu^{n+1}-\bu^n}{\tau}+\hf \left(\ext{\bu}\cdot\nabh\imp{\bhu}+\nabh\cdot\left(\imp{\bhu}(\ext{\bu})^{T}\right)\right)+\nabh p^n-\nu\Dh\imp{\bhu} \notag\\
	&\qquad =- \gamma \mAh\ext{\phi}\nabh\muhalf, \label{CN-NS}\\
	&\frac{\phi^{n+1}-\phi^n}{\tau}+\nabh\cdot\left(\mAh\ext{\phi}\imp{\bhu}\right)=\Dh\muhalf,	\label{CN-CH}\\
	&\muhalf=\frac{G(1+\phi^{n+1})-G(1+\phi^{n})}{\phi^{n+1}-\phi^{n}}+\frac{G(1-\phi^{n+1})-G(1-\phi^{n})}{\phi^{n+1}-\phi^{n}}-\theta_0\ext{\phi}-\epsilon^2\Dh\impp{\phi}\notag\\
	&\qquad\qquad+\tau\left( \mathcal{N}(\phi^{n+1}) - \mathcal{N}(\phi^n)  \right),	\label{CN-mu}\\
	&\frac{\bu^{n+1}-\bhu^{n+1}}{\tau}+ \frac12 \nabh ( 
	p^{n+1} -p^n )=0,	\label{CN-pressure-correction}\\
	&\nabh\cdot\bu^{n+1}=0,	\label{CN-incompressible}
\end{align}
where
\begin{equation}
	\begin{aligned}
		&\ext{\phi}:=\frac{3}{2}\phi^n-\hf \phi^{n-1},\quad\imp{\phi}:=\hf \phi^{n+1}+\hf \phi^{n},\quad\impp{\phi}:=\frac{3}{4}\phi^{n+1}+\frac{1}{4}\phi^{n-1},\\
		&\ext{\bu}:=\frac{3}{2}\bu^n-\hf \bu^{n-1},\quad\imp{\bhu}:=\hf \bhu^{n+1}+\hf \bu^{n},\quad G(x)=x\ln(x),
	\end{aligned}
\end{equation}
with the discrete boundary conditions:
\begin{equation}\label{physical-boundary-condition}
	( \bu^{n+1} \cdot \n ) |_{\Gamma}=0,\quad \partial_n ( \bu^{n+1} \cdot {\bf \tau} ) = 0 , \quad \partial_n\phi^{n+1}|_{\Gamma}=\partial_n\mu^{n+\hf }|_{\Gamma}=0.
\end{equation}
The nonlinear term $ \mathcal{N}(\phi^n)$ in \eqref{CN-mu} represents an increasing function within $\phi^n\in(-1,1)$, and possesses singularity at $\pm1$, for instance, we can choose $\mathcal{N}(\phi^n)=\ln(1+\phi^n)-\ln(1-\phi^n)$. 

At the initial time step, we could take a backward evaluation of the PDE system to obtain a locally second order accurate approximation to $\phi^{-1}$. In turn, a numerical implementation of the proposed algorithm \eqref{CN-NS}-\eqref{CN-incompressible} leads to a second order local truncation error at $n=0$.  


It is clear that the phase variable satisfies the mass conservation property, i.e.,
\begin{equation}
	\overline{\phi^{n+1}}=\overline{\phi^n}=\cdots=\overline{\phi^0}. \label{mass conserv-2} 
\end{equation}
To facilitate the later analysis, the following smooth function is introduced: for any $a>0$,
\begin{align}
	F_a(x):=\frac{G(x)-G(a)}{x-a},\ \forall x>0 . 
\end{align}
Therefore, the chemical potential approximation \eqref{CN-mu} can be represented as:
\begin{equation}
	\begin{aligned}
	\muhalf=&F_{1+\phi^n}(1+\phi^{n+1}) - F_{1-\phi^n}(1-\phi^{n+1})-\theta_0\ext{\phi}-\epsilon^2\Dh\impp{\phi}\\
	&\quad+\tau\left(\mathcal{N}(\phi^{n+1})-\mathcal{N}(\phi^n)   \right).
	\end{aligned}
\end{equation}
Moreover, the following results are useful in the convergence analysis: 

\begin{lm}\label{estimate of Fa} \cite{chen22a, chen24a} 
	Let $a>0$ be fixed, then
	\begin{enumerate}
		\item $F_a'(x)=\frac{G'(x)(x-a)-\left(G(x)-G(a)\right)}{(x-a)^2}\geq0$, for any $x>0$. 
		\item $F_a(x)$ is an increasing function of $x$, and $F_a(x)\leq F_a(a)=\ln a+1$ for any $0<x<a$.
		\item $F_a'(x)=\frac{1}{2\eta}$, for some $\eta$ between $a$ and $x$.
	\end{enumerate}
\end{lm}

The positivity-preserving property and unique solvability has been established in~\cite{chen24a}.

\begin{thm} \label{thm: positivity}  \cite{chen24a} 
		Given cell-centered functions $\phi^{n}$, $\phi^{n-1}$, with $\nrminf{\phi^n}$, $\nrminf{\phi^{n-1}}<1$ and $\overline{\phi^n}=\overline{\phi^{n-1}}=\beta_0<1$, then there exists a unique cell-centered solution $\phi^{n+1}$ to \eqref{CN-NS}-\eqref{CN-incompressible}, with $\nrminf{\phi^{n+1}}<1$, and $\overline{\phi^{n+1}}=\beta_0$.
\end{thm}

\begin{rem} 
The proposed chemical potential approximation~\eqref{CN-mu} is based on a modified Crank-Nicolson discretization, combined with a nonlinear artificial regularization. Such a nonlinear artificial regularization leads to an implicit treatment for the singular logarithmic terms, which in turn ensures the positivity-preserving and unique solvability properties. This technique has been widely used in various gradient flows, including the Cahn-Hilliard equation with Flory-Huggins potential~\cite{chen22a, chen19b, Dong2021a, Dong2022a, dong19b, dong20a, Qin2021, Yuan2021a, Yuan2022a}, the liquid film droplet model~\cite{ZhangJ2021}, the Poisson-Nernst-Planck system~\cite{LiuC2021a, LiuC2022a, LiuC2023a, qiao2014}, the reaction-diffusion system~\cite{LiuC2021b, LiuC2022c, LiuC2022b}, etc. The convex nature of the singular energy part prevents the numerical solution approach the singular limit values of $\pm1$, which turns out to be the key point in the theoretical arguments. 
\end{rem}

Now we proceed into the convergence analysis. For the CHNS system \eqref{equation-CHNS-1}-\eqref{equation-CHNS-4}, the existence of a global-in-time weak solution has been proved in~\cite{liu03}, under a polynomial approximation to the double well energy potential. The weak solution to the CHNS system with a Flory-Huggins phase field energy potential could be analyzed in the same fashion. Of course, the regularity of the weak solutions is not sufficient to justify the optimal rate convergence analysis, and a strong solution to the CHNS system, with higher order regularities, is needed in the error  estimate. It was proved in~\cite{abels09b} that unique strong solutions exist for the PDE system~\eqref{equation-CHNS-1}-\eqref{equation-CHNS-4}, with a constant mobility assumption. In more details, higher order regularities could be stated as follows: for any initial data $\bu_0 \in H^m (\Omega)$, $\phi_0 \in H^{m+1} (\Omega)$, there is an estimate for $\| \bu (t) \|_{H^m}$ and $\| \phi (t) \|_{H^{m+1}}$, $m \ge 1$, globally-in-time for 2D, and locally-in-time for 3D, under appropriate compatibility conditions between the initial data and boundary conditions \cite{temam01}. 
Therefore, for the exact solution $(\phi_e, \bu_e, p_e)$ to the FHCHNS system~\eqref{equation-CHNS-1}-\eqref{equation-CHNS-4},   we could always assume that the exact solution has regularity of class $\mathcal{R}$, with sufficiently regular initial data: 
\begin{equation}
\phi_e, \, \bu_e, \, p_e \in \mathcal{R} := H^4 \left(0,T; C_{\rm per}(\Omega)\right) \cap H^3 \left(0,T; C^2_{\rm per}(\Omega)\right) \cap L^\infty \left(0,T; C^6_{\rm per}(\Omega)\right).
	\label{assumption:regularity.1}
	\end{equation}
In addition, we assume that the following separation property is valid for the exact phase variable: for some $\delta >0$, 
	\begin{equation}
 - 1 + 2 \delta < \phi_e < 1 - 2 \delta , 
    \label{assumption:separation}
	\end{equation}
which is satisfied at a point-wise level, for all $t\in[0,T]$. Define $\Phi_N (\, \cdot \, ,t) := {\cal P}_N \phi_e (\, \cdot \, ,t)$, the (spatial) Fourier projection of the exact solution into ${\cal B}^K$, the space of trigonometric polynomials of degree up to and including $K$ (with $N=2K +1$), only in the Cosine wave mode in both the $x$ and $y$ directions, due to the homogeneous Neumann boundary condition.  The following projection approximation is standard: if $\phi_e \in L^\infty(0,T;H^\ell_{\rm per}(\Omega))$, for some $\ell\in\mathbb{N}$,
	\begin{equation}
\nrm{\Phi_N - \phi_e}_{L^\infty(0,T;H^k)}
   \le C h^{\ell-k} \nrm{\phi_e }_{L^\infty(0,T;H^\ell)},  \quad \forall \ 0 \le k \le \ell ,
    \, \, j= 1, 2. \label{projection-est-0}
	\end{equation}
By $\Phi_N^m$, $\phi_e^m$ we denote $\Phi_N (\, \cdot \, , t_m)$ and $\phi_e (\, \cdot \, , t_m)$, respectively, with $t_m = m\cdot \dt$. Since $\Phi_N \in {\cal B}^K$, the mass conservative property is available at the discrete level:
	\begin{equation}
\overline{\Phi_N^m} = \frac{1}{|\Omega|}\int_\Omega \, \Phi_N ( \cdot, t_m) \, d {\bf x} = \frac{1}{|\Omega|}\int_\Omega \, \Phi_N ( \cdot, t_{m-1}) \, d {\bf x} = \overline{\Phi_N^{m-1}} ,  \quad \forall \ m \in\mathbb{N}.
	\label{mass conserv-1}
	\end{equation}
On the other hand, the numerical solution of the phase variable is also mass conservative at the discrete level, as given by~\eqref{mass conserv-2}. Meanwhile, we use the mass conservative projection for the initial data:  $\phi^0 = {\mathcal P}_h \Phi_N (\, \cdot \, , t=0)$, that is
	\begin{equation}
\phi^0_{i,j} := \Phi_N (p_i, p_j, t=0)  .
	\label{initial data-0}
	\end{equation}	
In turn, the error grid function for the phase variable is defined as
	\begin{equation}
\tilde{\phi}^m := \mathcal{P}_h \Phi_N^m - \phi^m ,  \quad \forall \ m \in \left\{ 0 ,1 , 2, 3, \cdots \right\} .
	\label{error function-1}
	\end{equation}
Therefore, it follows that  $\overline{e^m} =0$, for any $m \in \left\{ 0 ,1 , 2, 3, \cdots \right\}$. Of course, the discrete norm $\nrm{ \, \cdot \, }_{-1,h}$ is well defined for the error grid function $e_\phi^m$. 

  For the velocity and pressure variables, we just take $\bU = \bu_e$ and $P = p_e$. Accordingly, the error grid functions for the velocity and pressure variables are defined as
	\begin{equation}
\tilde{\bu}^m := \mathcal{P}_h \bU^m - \bu^m = ( \tilde{u}^m , \tilde{v}^m)^T,  \, \, \, 
\tilde{p}^m := \mathcal{P}_h P^m - p^m, \quad \forall \ m \in \left\{ 0 ,1 , 2, 3, \cdots \right\} .
	\label{error function-2}
	\end{equation}
  
  The following theorem is the main result of this article.

\begin{thm}
	\label{thm: convergence}
Given initial data $\phi_e (\, \cdot \, ,t=0), \, \u_e (\, \cdot \, ,t=0) \in C^6_{\rm per}(\Omega)$, suppose the exact solution for FHCHNS system~\eqref{equation-CHNS-1}-\eqref{equation-CHNS-4} is of regularity class $\mathcal{R}$. Then, provided $\tau$ and $h$ are sufficiently small, and under the linear refinement requirement $C_1 h \le \tau \le C_2 h$, we have
	\begin{equation}
  \| \nabla_h \tilde{\phi}^n \|_2 +  \| \tilde{\bu}^n \|_2 
  + \Bigl( \frac{\epsilon^2}{8} \tau   \sum_{m=1}^{n} \| \nabla_h \Delta_h \tilde{\phi}^m \|_2^2 \Bigr)^{\hf} 
   \le C ( \tau^2+ h^2 ) ,
	\label{convergence-0}
	\end{equation}
for all positive integers $n$, such that $t_n=n\tau \le T$, where $C>0$ is independent of $\tau$ and $h$.
	\end{thm}

\section{Optimal rate convergence analysis} \label{sec:convergence} 
Throughout the following analysis, $C$ represents a constant that may depend on $\epsilon$, $\nu$, $\theta_0$, $\Omega$ and $\delta$, but is independent on $h$ and $\tau$.


\subsection{Higher order consistency analysis of \eqref{CN-NS}-\eqref{CN-incompressible}} 

By consistency, the projection solution $\Phi_N$, the exact profiles $\bU$ and $P$ solve the discrete equation~\eqref{CN-NS}-\eqref{CN-incompressible} with a second order (in both time and space) local truncation error. Meanwhile, it is observed that this leading local truncation error will not be enough to recover an a-priori $W_h^{1, \infty}$ bound for the numerical solution to recover the nonlinear analysis, as well as the separation property. To overcome this difficulty, we have to use a higher order consistency analysis, via a perturbation argument, to recover such a bound in later analysis. In more details, we construct a sequence of supplementary fields, and $\breve{\bU}$, $\breve{\Phi}$, $\breve{P}$, 
satisfying 
\begin{align}
&\breve{\textbf{U}}= \bm{\mathcal{P}}_H ( \textbf{U}+ \tau^2 \bU_{\tau,1} + h^2\bU_{h,1}),\ \breve{\Phi}=\Phi_N+\mathcal{P}_N(\tau^2 \Phi_{\tau,1} +h^2\Phi_{h,1}), \label{construction-1} \\
&\breve{P}= \bm{\mathcal{I}}_h ( P + \tau^2 P_{\tau,1} + h^2P_{h,1}),\quad 
\label{construction-2} 
\end{align} 
so that a higher $O (\tau^3 + h^4)$ consistency is satisfied with the given numerical scheme~\eqref{CN-NS}-\eqref{CN-incompressible}, in which $\bm{\mathcal{P}}_H$ stands for a discrete Helmholz interpolation (into the divergence-free space), and $\bm{\mathcal{I}}_h$ is the standard point-wise interpolation. 
The constructed fields $\bU_{\tau,1}$, $\Phi_{\tau,1}$,  $P_{\tau,1}$, 
$\bU_{h,1}$,  $\Phi_{h,1}$ 
and $P_{h,1}$, which will be obtained using a perturbation expansion, will depend solely on the exact solution $(\phi_e, \bu_e, p_e)$. 
In turn, we estimate the numerical error function between the constructed profile and the numerical solution, instead of a direct comparison between the numerical solution and projection solution. Such an $O (\tau^3 + h^4)$ truncation error enables us to derive a higher order convergence estimate, in the $\| \cdot \|_{H_h^1}$ norm for the phase variable and $\| \cdot \|_2$ norm for the velocity variable, which leads to a desired $\| \cdot \|_\infty$ bound of the numerical solution. 
This approach has been reported for a wide class of nonlinear PDEs; see the related works for the incompressible fluid equation~\cite{E95, E02, STWW03, STWW07, WL00, WL02, WLJ04}, various gradient equations~\cite{baskaran13b, guan17a, guan14a, LiX2021, LiX2023a, LiX2024a}, the porous medium equation based on the energetic variational approach~\cite{duan22a, duan20a}, nonlinear wave equation~\cite{WangL15}, etc. 

  The following bilinear form $b$ is introduced to facilitate the nonlinear analysis: 
\begin{align}
&b\left(\bu,\bv\right)= \bu\cdot\nabla \bv , \quad 
  b_h \left(\bu,\bv\right)=\frac{1}{2} ( \bu\cdot\nabh\bv+\nabh\cdot (\bu\bv^{T} ) ) . 
	\label{b}
\end{align}

The following intermediate velocity vector is defined, which is needed in the leading order consistency analysis: 
\begin{equation} 
   \hat{\bU}^{n+1} = \bU^{n+1} + \frac12 \tau \nabla ( P^{n+1} - P^n ) .  \label{consistency-hatu-1} 
\end{equation} 
A careful Taylor expansion in time reveals that 
\begin{align}
&\frac{\hat{\bU}^{n+1}-\bU^n}{\tau}+b(\ext{\bU},\bar{\hat{\bU}}^{n+\hf})+\nabla P^n -\nu\Delta \bar{\hat{\bU}}^{n+\hf} =-\gamma\ext{\Phi}_N\nabla\Muhalf\nonumber\\
&\qquad\qquad\qquad\qquad\qquad\qquad\qquad\qquad\qquad\qquad\qquad\qquad\quad+\tau^2 \mathbf{G}_0^{n+\frac{1}{2}}+O(\tau^3+h^{m_0}), \label{trun-NS}\\
&\frac{\Phi_N^{n+1}-\Phi_N^n}{\tau}+\nabla\cdot\left(\ext{\Phi}_N \bar{\hat{\bU}}^{n+\hf} \right)=\Delta\Muhalf+\tau^2 H_0^{n+\frac{1}{2}}+O(\tau^3+h^{m_0}),	\label{trun-CH}\\
&\Muhalf=\frac{G(1+\Phi_N^{n+1})-G(1+\Phi_N^{n})}{\Phi_N^{n+1}-\Phi_N^{n}}+\frac{G(1-\Phi_N^{n+1})-G(1-\Phi_N^{n})}{\Phi_N^{n+1}-\Phi_N^{n}}-\theta_0\ext{\Phi}_N-\epsilon^2\Delta\impp{\Phi}_N\notag\\
&\qquad\qquad+\tau\left( \mathcal{N}(\Phi^{n+1}_N)-\mathcal{N}(\Phi^n_N)   \right) ,  
  \label{trun-mu}\\
  &\frac{\bU^{n+1}-\hat{\bU}^{n+1}}{\tau}+ \frac12 \nabla ( 
	P^{n+1} - P^n )=0,	\label{trun-pressure-correction}\\   
&\nabla\cdot\bU^{n+1}=0,	\label{trun-incompressible}
\end{align}
where $\|\mathbf{G}_0^{n+\hf} \|$, $\|H_0^{n+\hf} \|$, $\|K_0\|\le C$, and $C$ depends on the regularity of the exact solutions.

The correction functions $\bU_{\tau,1}$, $\Phi_{\tau,1}$ 
and $P_{\tau,1}$ are solved by the following PDE system 
\begin{align*}
&	\partial_t\bU_{\tau,1}+(\bU_{\tau,1}\cdot\nabla)\bU+(\bU\cdot\nabla)\bU_{\tau,1}+\nabla P_{\tau,1}-\nu\Delta \bU_{\tau,1}=-\gamma\Phi_{\tau,1}\nabla \text{M}-\gamma\Phi_N \nabla \text{M}_{\tau,1}-\mathbf{G}_0,\\
&\partial_t\Phi_{\tau,1}+\nabla\cdot(\Phi_{\tau,1}\bU+\Phi_N \bU_{\tau,1})=\Delta\text{M}_{\tau,1}-H_0,\\
&\text{M}_{\tau,1}=\frac{\Phi_{\tau,1}}{1+\Phi_N}+\frac{\Phi_{\tau,1}}{1-\Phi_N}-\theta_0\Phi_{\tau,1}-\epsilon^2\Delta\Phi_{\tau,1} , \\  
&\nabla\cdot \bU_{\tau,1}=0 , 	
\end{align*}
combined with the homogeneous Neumann boundary condition for $\Phi_{\tau, 1}$ and $\text{M}_{\tau ,1}$, no-penetration, free-slip boundary condition for $\bU_{\tau ,1}$. Existence of a solution of the above linear PDE system is straightforward. Meanwhile, a similar intermediate velocity vector is introduced as  
\begin{equation} 
   \hat{\bU}_{\tau,1}^{n+1} = \bU_{\tau, 1}^{n+1} + \frac12 \tau \nabla ( P_{\tau, 1}^{n+1} - P_{\tau, 1}^n ) .  \label{consistency-hatu-2} 
\end{equation} 
In turn, an application of a temporal discretization to the above linear PDE system for $\bU_{\tau,1}$,  $\Phi_{\tau,1}$, 
$P_{\tau,1}$ and $\hat{\bU}_{\tau,1}$ indicates that
\begin{align}
&\frac{\hat{\bU}_{\tau,1}^{n+1}-\bU_{\tau,1}^n}{\tau}+b(\ext{\bU}_{\tau,1},\imp{\hat{\bU}})+b(\ext{\bU},\imp{\hat{\bU}}_{\tau,1})+\nabla P_{\tau,1}^n -\nu\Delta\imp{\hat{\bU}}_{\tau,1} \nonumber\\
&\qquad\qquad\qquad=-\gamma\ext{\Phi}_{\tau,1}\nabla\Muhalf-\gamma\ext{\Phi}_N\nabla\Muhalf_{\tau,1}-\mathbf{G}_0^{n+\frac{1}{2}}+O(\tau^2), \label{trun1-NS}\\
&\frac{\Phi_{\tau,1}^{n+1}-\Phi_{\tau,1}^n}{\tau}+\nabla\cdot\left(\ext{\Phi}_{\tau,1}\imp{\hat{\bU}}+\ext{\Phi}_N\imp{\hat{\bU}}_{\tau,1}\right)=\Delta\Muhalf_{\tau,1}- H_0^{n+\frac{1}{2}}+O(\tau^2),	\label{trun1-CH}\\
&\Muhalf_{\tau,1}=\frac{1}{\tau^2}\left(\frac{G(1+\Phi_N^{n+1}+\tau^2\Phi_{\tau,1}^{n+1})-G(1+\Phi_N^{n}+\tau^2\Phi_{\tau,1}^{n})}{\Phi_N^{n+1}+\tau^2\Phi_{\tau,1}^{n+1}-\Phi_N^{n}-\tau^2\Phi_{\tau,1}^{n}}-\frac{G(1+\Phi_N^{n+1})-G(1+\Phi_N^{n})}{\Phi_N^{n+1}-\Phi_N^{n}}\right)\nonumber\\
&\qquad\quad+\frac{1}{\tau^2}\left(\frac{G(1-\Phi_N^{n+1}-\tau^2\Phi_{\tau,1}^{n+1})-G(1-\Phi_N^{n}-\tau^2\Phi_{\tau,1}^{n})}{\Phi_N^{n+1}+\tau^2\Phi_{\tau,1}^{n+1}-\Phi_N^{n}-\tau^2\Phi_{\tau,1}^{n}}-\frac{G(1-\Phi_N^{n+1})-G(1-\Phi_N^{n})}{\Phi_N^{n+1}-\Phi_N^{n}}\right)\nonumber\\
&\qquad\quad-\theta_0\ext{\Phi}_{\tau,1}-\epsilon^2\Delta\impp{\Phi}_{\tau,1}+\tau\left( \mathcal{N}'(\Phi_N^{n+1})\Phi_{\tau,1}^{n+1}-\mathcal{N}'(\Phi_N^{n})\Phi_{\tau,1}^{n}  \right)  , 
  	\label{trun1-mu} \\
 &\frac{\bU_{\tau, 1}^{n+1}-\hat{\bU}_{\tau, 1}^{n+1}}{\tau}+ \frac12 \nabla ( 
	P_{\tau, 1}^{n+1} - P_{\tau, 1}^n )=0 , \label{trun 1-pressure-correction}\\ 
&\nabla\cdot\bU_{\tau,1}^{n+1}=0.	\label{trun1-incompressible}
\end{align}
It is noticed that a nonlinear Taylor expansion has been applied in the derivation of~\eqref{trun1-mu}. A combination of \eqref{trun-NS}-\eqref{trun-incompressible} and \eqref{trun1-NS}-\eqref{trun1-incompressible} leads to the following third order truncation error for $\breve{\textbf{U}}_1:=\bU+\tau^2 \bU_{\tau,1}$, $\breve{\Phi}_1:=\Phi_N+\tau^2\mathcal{P}_N\Phi_{\tau,1}$, 
$\breve{P}_{1}:=P+\tau^2 P_{\tau,1}$:
\begin{align}
&\frac{\breve{\hat{\bU}}_1^{n+1}-\breve{\bU}_1^n}{\tau}+b(\ext{\breve{\bU}}_1,\imp{\breve{\hat{\bU}}}_1)+\nabla \breve{P}_1^n -\nu\Delta\imp{\breve{\hat{\bU}}}_1 =-\gamma\ext{\breve{\Phi}}_1\nabla \breve{M}_1^{n+\frac12} \nonumber\\
&\qquad\qquad\qquad\qquad\qquad\qquad\qquad\qquad\qquad\qquad\qquad\qquad\quad+\tau^3 \mathbf{G}_1^{n+\frac{1}{2}}+O(\tau^4+h^{m_0}), \label{trun2-NS}\\
&\frac{\breve{\Phi}_1^{n+1}-\breve{\Phi}_1^n}{\tau}+\nabla\cdot\left(\ext{\breve{\Phi}}_1 \imp{\breve{\hat{\bU}}}_1\right)=\Delta \breve{M}_1^{n+\frac12} +\tau^3 H_1^{n+\frac{1}{2}}+O(\tau^4+h^{m_0}),	\label{trun2-CH}\\
& \breve{M}_1^{n+\frac12} =\frac{G(1+\breve{\Phi}_1^{n+1})-G(1+\breve{\Phi}_1^{n})}{\breve{\Phi}_1^{n+1}-\breve{\Phi}_1^{n}}+\frac{G(1-\breve{\Phi}_1^{n+1})-G(1-\breve{\Phi}_1^{n})}{\breve{\Phi}_1^{n+1}-\breve{\Phi}_1^{n}}-\theta_0\ext{\breve{\Phi}}_1 
 -\epsilon^2\Delta\impp{\breve{\Phi}}_1\notag\\
&\qquad\qquad+\tau\left( \mathcal{N}(\breve{\Phi}_1^{n+1})-\mathcal{N}(\breve{\Phi}_1^n)     \right) , 
\label{trun2-mu}\\
&\frac{\breve{\bU}_1^{n+1}-\breve{\hat{\bU}}_1^{n+1}}{\tau}+ \frac12 \nabla ( 
	\breve{P}_1^{n+1} - \breve{P}_1^n )=0 , \label{trun2-pressure-correction}\\ 
&\nabla\cdot\breve{\bU}_1^{n+1}=0,	\label{trun2-incompressible}
\end{align}
where $\|\mathbf{G}_1\|\le C$, $\|H_1\|\le C$, 
and $C$ depends on the regularity of the exact solutions. The following linearized expansions have been used in \eqref{trun2-mu}:
\begin{align}
&	\mathcal{N}(\breve{\Phi}_1)=\mathcal{N}(\Phi_N+\tau^2\mathcal{P}_N\Phi_{\tau,1})=\mathcal{N}(\Phi_N)+\mathcal{N}'(\Phi_N)\tau^2\mathcal{P}_N\Phi_{\tau,1}+O(\tau^4).\label{linearied expan2}
\end{align}

Now, we construct the spatial correction function to upgrade the spatial accuracy order.  In terms of the spatial discretization, it is observed that the velocity profile $\breve{\bU}_1$ is not divergence-free at a discrete level, so that its discrete inner product with the pressure gradient may not vanish. To overcome the difficulty, a spatial interpolation operator is needed to ensure the exact divergence-free property of the constructed velocity vector at a discrete level. Such an operator in the finite difference discretization is highly non-standard, due to the collocation point structure, and this effort has not been reported in the existing textbook literature. A pioneering idea of this approach was proposed in~\cite{E1996a}, and other related analysis works have been reported in~\cite{chen16,  cheng2023a, STWW03}, etc.   

In more details, the spatial interpolation operator $\bm{\mathcal{P}}_H$ is defined as follows, for any $\bu \in H^1(\Omega)$, $\nabla\cdot \bu=0$: There is an exact stream function $\psi$ so that $\bu=\nabla^{\perp} \psi$, and we define 
\begin{equation}
\bm{\mathcal{P}}_H (\bu)=\nabla_h^{\perp} \psi =( - D_y \psi , D_x \psi )^T . 
\end{equation}
Of course, this definition ensures $\nabla_h \cdot \bm{\mathcal{P}}_H (\bu)=0$ at a point-wise level. Furthermore, an $O (h^2)$ truncation error is available between the continuous velocity vector $\bu$ and its Helmholtz interpolation, $\bm{\mathcal{P}}_H (\bu)$. 

In turn, we denote $\breve{\bU}_{1, PH} = \bm{\mathcal{P}}_H (\breve{\bU}_1)$. An application of the spatial discretization yields the following truncation error estimate, with the help of straightforward Taylor expansion:
\begin{align}
&\frac{\breve{\hat{\bU}}_{1, PH}^{n+1}-\breve{\bU}_{1, PH}^n}{\tau}+b_h (\ext{\breve{\bU}}_{1, PH},\imp{\breve{\hat{\bU}}}_{1, PH})+\nabh \breve{P}_1^n -\nu\Dh\imp{\breve{\hat{\bU}}}_{1, PH} =-\gamma\mAh\ext{\breve{\Phi}}_1 \nabh \breve{\text{M}}_{1, h}^{n+\frac12} \nonumber\\
&\qquad\qquad\qquad\qquad\qquad\qquad\qquad\qquad\qquad\qquad\qquad\qquad\quad+h^2 \mathbf{G}_h^{n+\frac{1}{2}}+O(\tau^3+h^4), \label{trun4-NS}\\
&\frac{\breve{\Phi}_1^{n+1}-\breve{\Phi}_1^n}{\tau}+\nabh\cdot\left(\mAh\ext{\breve{\Phi}}_1 \imp{\breve{\hat{\bU}}}_{1, PH} \right)=\Dh \breve{\text{M}}_{1, h}^{n+\frac12} +h^2 H_h^{n+\frac{1}{2}}+O(\tau^3+h^4),	\label{trun4-CH}\\
& \breve{\text{M}}_{1, h}^{n+\frac12} =\frac{G(1+\breve{\Phi}_1^{n+1})-G(1+\breve{\Phi}_1^{n})}{\breve{\Phi}_1^{n+1}-\breve{\Phi}_1^{n}}+\frac{G(1-\breve{\Phi}_1^{n+1})-G(1-\breve{\Phi}_1^{n})}{\breve{\Phi}_1^{n+1}-\breve{\Phi}_1^{n}}-\theta_0 \ext{\breve{\Phi}}_1-\epsilon^2\Delta_h \impp{\breve{\Phi}}_1 \notag\\
&\qquad\qquad+\tau\left(  \mathcal{N}(\breve{\Phi}_1^{n+1})-\mathcal{N}(\breve{\Phi}_1^n)     \right) , 
\label{trun4-mu}\\
&\frac{\breve{\bU}_{1, PH}^{n+1}-\breve{\hat{\bU}}_{1, PH}^{n+1}}{\tau}+ \frac12 \nabla_h ( 
	\breve{P}_1^{n+1} - \breve{P}_1^n )=0 ,  \label{trun4-pressure-correction}\\ 
&\nabh\cdot\hat{\bU}_{1, PH}^{n+1}=0,	\label{trun4-incompressible}
\end{align}
where $\|\mathbf{G}_h\|_2$, $\|H_h\|_2 \le C$, 
dependent only on the regularity of the exact solution. Meanwhile, it is noticed that there is no $O (h^3)$ truncation error term, which comes from the fact that the centered difference spatial approximation gives local truncation errors with only even order terms, $O (h^2)$, $O (h^4)$, etc., because of the Taylor expansion symmetry over a uniform mesh. This fact will greatly simplify the construction of the higher order consistency analysis in the spatial discretization.  

The spatial correction functions $\bU_{h,1}$, $\Phi_{h,1}$ 
and  $P_{h,1}$ are determined by the following system 
\begin{align*}
&	\partial_t\bU_{h,1}+(\bU_{h,1}\cdot\nabla)\breve{\bU}_1+(\breve{\bU}_1\cdot\nabla)\bU_{h,1}+\nabla P_{h,1}-\nu\Delta \bU_{h,1}=-\gamma\Phi_{h,1}\nabla \breve{\text{M}}_1 -\gamma\breve{\Phi}_1 \nabla \text{M}_{h,1}-\mathbf{G}_h,\\
&\partial_t\Phi_{h,1}+\nabla\cdot(\Phi_{h,1}\breve{\bU}_1+\breve{\Phi}_1 \bU_{h,1})=\Delta\text{M}_{h,1}-H_h,\\
&\text{M}_{h,1}=\frac{\Phi_{h,1}}{1+\breve{\Phi}_1}+\frac{\Phi_{h,1}}{1-\breve{\Phi}_1}-\theta_0\Phi_{h,1}-\epsilon^2\Delta\Phi_{h,1} , \\ 
&\nabla\cdot \bU_{h,1}=0.	
\end{align*}
Again, the homogeneous Neumann boundary condition is imposed for $\Phi_{h, 1}$ and $\text{M}_{h ,1}$, no-penetration, free-slip boundary condition for $\bU_{h ,1}$. 
Subsequently, we denote $\bU_{h,1, PH} = \bm{\mathcal{P}}_H (\bU_{h,1})$, and $\hat{\bU}_{h,1, PH}^{n+1} = \bU_{h,1, PH}^{n+1} + \frac12 \tau \nabla_h ( P_{h,1}^{n+1} - P_{h,1}^n )$. An application of both the temporal and spatial approximations to the above system reveals that 
\begin{align}
&\frac{\hat{\bU}_{h,1, PH}^{n+1}-\bU_{h,1, PH}^n}{\tau}+b_h (\ext{\bU}_{h,1, PH},\imp{\hat{\bU}}_{1, PH})+ b_h (\ext{\breve{\bU}}_{1, PH},\imp{\bU}_{h,1, PH})+\nabh P_{h,1}^n -\nu\Dh\imp{\hat{\bU}}_{h,1, PH} \nonumber\\
&\qquad\qquad\qquad=-\gamma\mAh\ext{\Phi}_{h,1}\nabh \breve{\text{M}}_{1, h}^{n+\frac12} -\gamma\mAh\ext{\breve{\Phi}}_1 \nabh\Muhalf_{h,1}-\mathbf{G}_h^{n+\frac{1}{2}}+O(\tau^2+h^2), \label{trun5-NS}\\
&\frac{\Phi_{h,1}^{n+1}-\Phi_{h,1}^n}{\tau}+\nabh\cdot\left(\mAh\ext{\Phi}_{h,1}\imp{\hat{\bU}}_{1, PH} +\mAh\ext{\breve{\Phi}}_1 \imp{\bU}_{h,1, PH}\right)=\Dh\Muhalf_{h,1}- H_h^{n+\frac{1}{2}}+O(\tau^2+h^2),	\label{trun5-CH}\\
&\Muhalf_{h,1}=\frac{1}{h^2}\left(\frac{G(1+\breve{\Phi}_1^{n+1}+h^2\Phi_{h,1}^{n+1})-G(1+\breve{\Phi}_1^{n}+h^2\Phi_{h,1}^{n})}{\breve{\Phi}_1^{n+1}+h^2 \Phi_{h,1}^{n+1}-\breve{\Phi}_1^{n}-h^2\Phi_{h,1}^{n}}-\frac{G(1+\breve{\Phi}_1^{n+1})-G(1+\breve{\Phi}_1^{n})}{\breve{\Phi}_1^{n+1}-\breve{\Phi}_1^{n}}\right)\nonumber\\
&\qquad\quad+\frac{1}{h^2}\left(\frac{G(1-\breve{\Phi}_1^{n+1}-h^2\Phi_{h,1}^{n+1})-G(1-\breve{\Phi}_1^{n}-h^2\Phi_{h,1}^{n})}{\breve{\Phi}_1^{n+1}+h^2\Phi_{h,1}^{n+1}-\breve{\Phi}_1^{n}-h^2\Phi_{h,1}^{n}}-\frac{G(1-\breve{\Phi}_1^{n+1})-G(1-\breve{\Phi}_1^{n})}{\breve{\Phi}_1^{n+1}-\breve{\Phi}_1^{n}}\right)\nonumber\\
&\qquad\quad-\theta_0\ext{\Phi}_{h,1}-\epsilon^2\Dh\impp{\Phi}_{h,1}+\tau\left(\mathcal{N}'(\breve{\Phi}_1^{n+1})\Phi_{h,1}^{n+1}-\mathcal{N}'(\breve{\Phi}_1^{n})\Phi_{h,1}^{n}    \right) , 
\label{trun5-mu}\\
&\frac{\bU_{h,1, PH}^{n+1}- \hat{\bU}_{h,1, PH}^{n+1}}{\tau}+ \frac12 \nabla_h ( 
	P_{h,1}^{n+1} - P_{h,1}^n )=0 ,  \label{trun5-pressure-correction}\\ 
&\nabh\cdot\bU_{h,1, PH}^{n+1}=0.	\label{trun5-incompressible}
\end{align}
A combination of \eqref{trun4-NS}-\eqref{trun4-incompressible} and \eqref{trun5-NS}-\eqref{trun5-incompressible} leads to the following $O(\tau^3+h^4)$ truncation error for $\breve{\textbf{U}}$, $\breve{\Phi}$ 
and $\breve{P}$:
\begin{align}
&\frac{\hat{\breve{\bU}}^{n+1}-\breve{\bU}^n}{\tau}+b_h (\ext{\breve{\bU}},\imp{\hat{\breve{\bU}}})+\nabh \breve{P}^n -\nu\Dh\imp{\hat{\breve{\bU}}} =-\gamma\mAh\ext{\breve{\Phi}}\nabh \breve{\text{M}}^{n+\frac12} +\bm{\zeta}_1^{n}, \label{trun6-NS}\\
&\frac{\breve{\Phi}^{n+1}-\breve{\Phi}^n}{\tau}+\nabh\cdot\left(\mAh\ext{\breve{\Phi}}\imp{\hat{\breve{\bU}}}\right)=\Dh \breve{\text{M}}^{n+\frac12} +\zeta_2^{n},	\label{trun6-CH}\\
&\breve{\mbox{M}}^{n+\frac{1}{2}}=\frac{G(1+\breve{\Phi}^{n+1})-G(1+\breve{\Phi}^{n})}{\breve{\Phi}^{n+1}-\breve{\Phi}^{n}}+\frac{G(1-\breve{\Phi}^{n+1})-G(1-\breve{\Phi}^{n})}{\breve{\Phi}^{n+1}-\breve{\Phi}^{n}}-\theta_0\ext{\breve{\Phi}}-\epsilon^2\Dh\impp{\breve{\Phi}}\notag\\
&\qquad\qquad+\tau\left( \mathcal{N}(\breve{\Phi}^{n+1})-\mathcal{N}(\breve{\Phi}^n)  \right) , 
\label{trun6-mu}\\
&\frac{\breve{\bU}^{n+1}- \hat{\breve{\bU}}^{n+1}}{\tau}+ \frac12 \nabla_h ( 
	\breve{P}^{n+1} - \breve{P}^n )=0 ,  \label{trun6-pressure-correction}\\ 
&\nabh\cdot\breve{\bU}^{n+1}=0,	\label{trun6-incompressible}
\end{align}
where similar linearized expansions as in 
\eqref{linearied expan2} have been utilized in \eqref{trun6-mu}, and the truncation errors $\|\bm{\zeta}_1^{n} \|_2=O(\tau^3+h^4+\tau^2h^2)$, $\zeta_2^{n}:=O(\tau^3+h^4+\tau^2h^2)$. Notice that $\tau^2h^2\le \frac{1}{2}\tau^4+\frac{1}{2}h^4$, then we obtain
\begin{align}
	\|\bm{\zeta}_1^{n}\|_2 \le C(\tau^3+h^4),\quad \|\zeta_2^{n}\|_2 \le C(\tau^3+h^4) .
\end{align} 

\begin{rem}
	For the sake of the later analysis, we could set the initial value as $\Phi_{\tau,1}(\cdot,t=0)\equiv0$ and $\Phi_{h,1}(\cdot,t=0)\equiv0$, respectively. Then we obtain $\frac{\mathrm{d}}{\mathrm{d}t}(\Phi_{\tau,1},1)=0$ and $\frac{\mathrm{d}}{\mathrm{d}t}(\Phi_{h, 1},1)=0$, so that $(\Phi_{\tau, 1} (t),1)=(\Phi_{\tau, 1} (0),1)=0$, $(\Phi_{h, 1} (t),1)=(\Phi_{h, 1} (0),1)=0$. Moreover, by the construction formula~\eqref{construction-1}, it is clear that $\overline{\breve{\Phi}^n}=\overline{\Phi_N^n}$, due to the fact that $\overline{{\mathcal P}_N (f)} = (f , 1)$. This in turn implies the mass conservative property for $\breve{\Phi}$: 
\begin{align}
&
\phi^0 \equiv  \breve{\Phi}^0 ,  \quad \overline{\phi^n} = \overline{\phi^0} ,  \quad \forall \, n \ge 0,
\label{mass conserv-3-1}
\\
 &
\overline{\breve{\Phi}^n} = \overline{\Phi_N^n} = \int_\Omega \, \Phi_N^n \, d {\bf x}  
= \int_\Omega \, \Phi_N^0 \, d {\bf x} = \overline{\phi^0} = \overline{\phi^n} ,  \quad \forall \, n \ge 0,  
	\label{mass conserv-3-2}
\end{align}
where the second step is based on the fact that $\Phi_N \in {\cal B}^K$, and the third step comes from the mass conservative property of $\Phi_N$ at the continuous level. These two properties will be used in the later analysis. In addition, since $\breve{\Phi}$ is mass conservative at a discrete level, we observe that the local truncation error $\tau_\phi$ has a similar property: 
\begin{equation} 
  \overline{\zeta_2^{n}} = 0 ,  \quad \forall \, n \ge 0 . \label{mass conserv-3-3}
\end{equation} 
\end{rem}

\begin{rem}
  We recall the phase separation inequality~\eqref{assumption:separation} for the exact solution. Furthermore, by choosing sufficiently small $\tau$ and $h$ satisfying $h\le\tau_0,\ \tau\le \tau_0$, the following separation property becomes available for the constructed profile $\breve{\Phi}$:
	\begin{align}  \label{separation-1} 
		-1+\frac{3}{2}\delta\le \breve{\Phi} \le 1-\frac{3}{2}\delta.
	\end{align} 
	
Since the correction function is only based on the exact solution $(\Phi, \bU, P)$ with enough regularity, its discrete $W_h^{1,\infty}$ norm will stay bounded, for any $k , n \ge 0$: 
\begin{equation} 
  \| \breve{\Phi}^k \|_\infty \leq C^\star , \, \, \, 
  \| {\breve{\bU}}^k \|_\infty \leq C^\star   , \, \, \, 
  \| \nabla_h \breve{\Phi}^k \|_\infty \leq C^\star , \, \, \, 
  \| \nabla_h \breve{\bU}^k \|_\infty \leq C^\star  , \, \, \, 
  \| \imp{\hat{\breve{\bU}}} \|_\infty \le C^* .   
    \label{assumption:W1-infty bound}
\end{equation}  
In addition, it is clear that $\breve{\text{M}}^{n+\frac12}$ only depends on the exact solution $\Phi$ and the associated correction functions, we assume a discrete $W_h^{1,\infty}$ bound
\begin{equation}
\| \nabla_h \breve{{M}}^{n+\frac12}\|_\infty \leq C^*.  \label{assumption:W1-infty bound mu}
\end{equation}

\end{rem}

\subsection{A rough error estimate}

The following error functions are defined: 
\begin{equation} 
	\begin{aligned}
		&\eu^n = \breve{\bU}^n-\bu^n,\, \, \, \, &&\beu=\imp{\breve{\bU}}-\imp{\bu}, \, \, \, &&\teu=\ext{\breve{\bU}}-\ext{\bu}, \\  
		&\ehu^{n+1} = \hat{\breve{\bU}}^{n+1} - \hat{\bu}^{n+1} , && 
		\behu = \frac12 (\eu^n + \ehu^{n+1} ) , \, \, \, 
		&& \\
		&e_p^n = \breve{P}^n-p^n,\, \, \,  &&\bep=\imp{\breve{P}}-\imp{p},\, \, \, &&\emu= \breve{\text{M}}^{n+\frac12} -\muhalf,\\
		&\ephi^n= \breve{\Phi}^n-\phi^n,\, \, \, &&\bbephi=\impp{\breve{\Phi}}-\impp{\phi},\, \, \, && \tephi=\ext{\breve{\Phi}}-\ext{\phi}.
	\end{aligned} 
	\label{error function-3} 
\end{equation} 
Notice that these error functions, $(e_\phi^n, \eu^n, e_p^n)$, are different from the earlier introduced ones, $( \tilde{\phi}^n, \tilde{\bu}^n, \tilde{p}^n)$ (given by~\eqref{error function-1}, \eqref{error function-2}). In fact, the latter ones measure a direct difference between the numerical and exact (projection) solutions, while the ones in \eqref{error function-3} correspond to the difference between the numerical solution and the constructed profiles. Due to the higher order consistency analysis for $(\breve{\Phi}, \breve{\bU}, \breve{P})$, an $O (\tau^3 + h^4)$ error estimate would be available for the ones defined in \eqref{error function-3}, while the original numerical error functions preserve an $O (\tau^2 + h^2)$ accuracy order, as indicated by Theorem~\ref{thm: convergence}. Based on the mass conservative identity~\eqref{mass conserv-3-2}, we see that the error function $\ephi^k$ is always mean free: $\overline{\ephi^k} =0$, for any $k \ge 0$, so that its $\| \cdot \|_{-1, h}$ norm is well defined. 


\begin{lm}\label{B-est}  A trilinear form is introduced as $\B{\bu}{\bv}{\textbf{w}}=\eipvec{b_h \left(\bu,\bv\right)}{\textbf{w}}_1$. The following estimates are valid: 
	\begin{align} 
	     & 
	     \B{\bu}{\bv}{\bv} = 0 ,  \label{B est-0-1} 
\\
	      & 
		|\B{\bu}{\bv}{\textbf{w}}| \leq \frac12 \nrm{\bu}_2 \left(\nrm{\nabh \bv}_\infty \cdot \| \textbf{w} \|_2 +\nrm{\nabh \textbf{w}}_2 \cdot \nrminf{\bv} \right) . \label{B est-0-2} 
	\end{align}
\end{lm} 

\begin{proof} 
Identity~\eqref{B est-0-1} comes from the summation by parts formula 
\begin{equation*} 
   \B{\bu}{\bv}{\bv}  
   = \eipvec{b_h \left(\bu,\bv\right)}{\bv}_1 = \frac12 ( \langle \bu\cdot \nabh \bv , \bv \rangle_1 
   + \langle \nabh\cdot (\bu\bv^{T} ) , \bv \rangle_1 ) = 0 . 
\end{equation*} 
Inequality~\eqref{B est-0-2} could be derived as follows:   
\begin{equation*} 
\begin{aligned}
  | \B{\bu}{\bv}{\textbf{w}}| = & \frac12 | \langle \bu\cdot \nabh \bv ,  \textbf{w} \rangle_1 
   + \langle \nabh\cdot (\bu\bv^{T} ) , \textbf{w} \rangle_1 ) | 
   = \frac12 | \langle \bu\cdot \nabh \bv ,  \textbf{w} \rangle_1 
   - \langle \bu \cdot \nabla_h \textbf{w} , \bv \rangle_1 ) | 
\\
  \le & 
  \frac12 ( | \langle \bu\cdot \nabh \bv ,  \textbf{w} \rangle_1 | 
   + | \langle \bu \cdot \nabla_h \textbf{w} , \bv \rangle_1 ) | ) 
   \le \frac12 \nrm{\bu}_2 \left(\nrm{\nabh \bv}_\infty \cdot \| \textbf{w} \|_2 
  +\nrm{\nabh \textbf{w}}_2 \cdot \nrminf{\bv} \right) . 
\end{aligned} 
\end{equation*} 
\end{proof} 

In addition, for the nonlinear error terms associated with the singular logarithmic terms, the following preliminary estimates will be needed in the convergence analysis; the detailed proofs will be given in Appendix. 

\begin{prop} \label{prop: nonlinear est-1} 
Given $\Psi$ and $\psi$ and denote $\tilde{\psi} = \Psi - \psi$. Assume that both $\Psi$ and $\psi$ preserve a separation property: 
\begin{equation} 
  -1 + \delta \le \Psi , \, \psi \le 1 - \delta  , \label{separation-2} 
\end{equation} 
then the following nonlinear error estimates are available: 
\begin{align} 
  \| \ln ( 1 \pm \Psi) - \ln ( 1 \pm \psi ) \|_2 \le C_\delta \| \tilde{\psi} \|_2 , \quad 
  \| {\mathcal N} ( \Psi ) - {\mathcal N} ( \psi ) \|_2 \le C_\delta \| \tilde{\psi} \|_2 . 
  \label{nonlinear est-0-1} 
\end{align} 
Moreover, under a few additional conditions 
\begin{equation} 
  \| \nabla_h \Psi \|_\infty \le C^*, \, \, \, \| \nabla_h \psi \|_\infty \le \tilde{C}_1 := C^* + 1 , \, \, \, 
  \| \tilde{\psi} \|_\infty \le \tau + h , 
  \label{nonlinear est-condition-1} 
\end{equation} 
we have the further estimates 
\begin{equation} 
\begin{aligned} 
  & 
  \| \nabla_h ( \ln ( 1 \pm \Psi) - \ln ( 1 \pm \psi ) ) \|_2 
  \le C_\delta ( \| \tilde{\psi} \|_2  +  \| \nabla_h \tilde{\psi} \|_2 ) , 
\\
  & 
  \| \nabla_h ( {\mathcal N} ( \Psi ) - {\mathcal N} ( \psi ) ) \|_2 
  \le C_\delta ( \| \tilde{\psi} \|_2  +  \| \nabla_h \tilde{\psi} \|_2 ) ,  
\end{aligned} 
  \label{nonlinear est-0-2} 
\end{equation} 
in which $C_\delta$ only depends on $\delta$ and $C^*$, independent on $\tau$ and $h$. 
\end{prop} 

\begin{prop} \label{prop: nonlinear est-2} 
Assume that the numerical solution $\phi^n$ preserves the separation property, $-1 + \delta \le \phi^n \le 1 - \delta$, as in~\eqref{separation-2}, then the following nonlinear error estimates are valid: 
\begin{equation} 
\begin{aligned} 
  & 
   \langle \ephi^{n+1} , F_{1 + \breve{\Phi}^n}(1 + \breve{\Phi}^{n+1}) - F_{1 + \phi^n}(1+\phi^{n+1}) \rangle_c \ge - C_\delta \| \ephi^{n+1} \|_2 \cdot \| \ephi^n \|_2 , 
\\
  & 
  \langle \ephi^{n+1} , - F_{1 - \breve{\Phi}^n}(1 - \breve{\Phi}^{n+1}) + F_{1 - \phi^n}(1-\phi^{n+1}) \rangle_c \ge - C_\delta \| \ephi^{n+1} \|_2 \cdot \| \ephi^n \|_2 .  
\end{aligned} 
  \label{nonlinear est-1-1} 
\end{equation} 
Moreover, if both $\phi^n$ and $\phi^{n+1}$ preserve the separation property, $-1 + \delta \le \phi^k \le 1 - \delta$, $k=n, n+1$, and the following conditions are satisfied: 
\begin{equation} 
  \| \nabla_h \phi^k \|_\infty \le \tilde{C}_1 := C^* + 1 , \, \, \, 
  \| \ephi^k \|_\infty \le \tau + h ,  \quad k = n, n+1 , 
  \label{nonlinear est-condition-2} 
\end{equation} 
we have the further estimates 
\begin{equation} 
\begin{aligned} 
  & 
  \| \nabla_h ( F_{1 + \breve{\Phi}^n}(1 + \breve{\Phi}^{n+1}) 
   - F_{1 + \phi^n}(1+\phi^{n+1}) ) \|_2 
  \le C_\delta (   \| \nabla_h \ephi^n \|_2 
  +  \| \nabla_h \ephi^{n+1} \|_2 ) , 
\\
  & 
  \| \nabla_h ( F_{1 - \breve{\Phi}^n}(1 - \breve{\Phi}^{n+1}) 
   - F_{1 - \phi^n}(1-\phi^{n+1}) ) \|_2 
  \le C_\delta (  \| \nabla_h \ephi^n \|_2 
   +  \| \nabla_h \ephi^{n+1} \|_2 ) , 
\end{aligned} 
  \label{nonlinear est-1-2} 
\end{equation} 
in which $C_\delta$ only depends on $\delta$ and $C^*$, independent on $\tau$ and $h$. 
\end{prop}

Subtracting the numerical system \eqref{CN-NS}-\eqref{CN-incompressible} from the consistency estimate \eqref{trun6-NS}-\eqref{trun6-incompressible} leads to the following error equations:
\begin{align}
	&\frac{\ehu^{n+1}-\eu^n}{\tau}+b_h ( \teu,\imp{\hat{\breve{\bU}}}) + b_h (\ext{\bu}, \behu) +\nabh e_p^n -\nu\Dh \behu \notag\\
	&\qquad +\gamma\mAh \tephi \nabh \breve{\text{M}}^{n+\frac12} + \gamma\mAh\ext{\phi}\nabh \emu =\bm{\zeta}_1^n, \label{error-NS}\\
	&\frac{\ephi^{n+1}-\ephi^n}{\tau}+\nabh\cdot\left(\mAh \tephi \imp{\hat{\breve{\bU}}} + \mAh\ext{\phi} \behu \right)=\Dh\emu + \zeta_2^n, \label{error-CH}\\
	&\emu=F_{1+\breve{\Phi}^n}(1+\breve{\Phi}^{n+1}) - F_{1+\phi^n}(1+\phi^{n+1}) 
	 - F_{1-\breve{\Phi}^n}(1-\breve{\Phi}^{n+1}) + F_{1-\phi^n}(1-\phi^{n+1})\notag\\
	&\qquad\qquad -\theta_0\tephi -\epsilon^2 \Dh\bbephi +\tau \left(\mathcal{N}(\breve{\Phi}^{n+1}) -\mathcal{N}(\phi^{n+1}) -\mathcal{N}(\breve{\Phi}^n) +\mathcal{N}(\phi^n)\right) ,	\label{error-mu}\\
&\frac{\eu^{n+1}- \ehu^{n+1}}{\tau}+ \frac12 \nabla_h ( 
	e_p^{n+1} - e_p^n )=0 ,  \label{error-pressure-correction}\\ 		
	&\nabh\cdot\eu^{n+1}=0.	\label{error-incompressible}
\end{align}
To proceed with the convergence analysis, the following a-priori assumption is made for the numerical error functions at the previous time steps: 
\begin{equation} \label{a priori-1}
	\| \eu^k \|_2 , \, \| \nabla_h \ephi^k \|_2 \leq \tau^\frac{11}{4} + h^\frac{15}{4} ,  \, \, \, k= n, n-1 , 
	\quad 
	\| \nabla_h e_p^n \|_2  \le \tau^\frac74 + h^\frac{11}{4}  . 
\end{equation} 
Such an a-priori assumption will be recovered by the convergence analysis in the next time step,  which will be demonstrated later. In particular, this induction assumption is valid for $n=0$ and $n=1$. At $n=0$, the initial data is given by a point-wise interpolation of the projection solution, so that the numerical error is of order $O (h^m)$. At the next time step $n=1$, since the second order local truncation error is obtained in the numerical implementation of \eqref{CN-NS}-\eqref{CN-incompressible} from $n=0$ to $n=1$, the higher order consistency analysis outlined in the last section would also be applicable, so that an $O (\tau^3 + h^4)$ consistency is available for the defined numerical error functions.  

In turn, the a-priori assumption~\eqref{a priori-1} leads to a $W_h^{1, \infty}$ bound for the numerical error function at the previous time steps, based on the inverse inequality, the linear refinement requirement $C_1h \leq \tau \leq C_2h$, as well as the discrete Poincar\'e inequality (stated in Proposition~\ref{prop: Poincare}): 
\begin{equation} 
\begin{aligned}  
  & 
\| \ephi^k \|_\infty \leq \frac{C\| \ephi^k \|_{H^1_h}}{h^\frac{1}{2}} \leq C (\tau^\frac94 + h^\frac{13}{4})  \le \tau^2 + h^3 \leq \frac{\delta}{2} ,  
\\
  & 
  \| \nabla_h \ephi^k \|_\infty \leq \frac{C \| \ephi^k \|_\infty}{h} \leq C (\tau^\frac54 + h^\frac94) 
\le \tau + h^2 \le 1 ,  
\end{aligned}   
 \label{a priori-2} 
\end{equation} 
for $k=n, n-1$, provided that $\tau$ and $h$ are sufficiently small. In turn, the phase separation property becomes available for $\phi^n$, in combination with the separation estimate~\eqref{separation-1} for the constructed profile $\breve{\Phi}$:
	\begin{equation}  \label{separation-3} 
		-1+ \delta \le \phi^n \le 1- \delta .
	\end{equation} 
Furthermore, a $W_h^{1, \infty}$ bound of the numerical solution could also be derived: 
\begin{align} 
   &  
  \| \tilde{\phi}^{n+\frac12} \|_\infty = \| \frac32 \phi^n - \frac12 \phi^{n-1} \|_\infty 
  \le \frac32 \| \phi^n \|_\infty + \frac12 \| \phi^{n-1} \|_\infty \le 2 ,  \label{a priori-3-1} 
\\
  & 
  \| \nabla_h \phi^n \|_\infty \le \| \nabla_h \breve{\Phi}^n \|_\infty 
  + \| \nabla_h \ephi^n \|_\infty \le C^* +1 = \tilde{C}_1 ,  
  \label{a priori-3-2} 
\end{align} 
in which the $W_h^{1, \infty}$ assumption~\eqref{assumption:W1-infty bound} for the constructed solution $\breve{\Phi}$ has been applied. 

Taking a discrete inner product with \eqref{error-NS} by $2 \behu = \ehu^{n+1} + \eu^n$ gives 
\begin{equation} 
\begin{aligned} 
	&
	\frac{1}{\tau} ( \| \ehu^{n+1} \|_2^2 - \| \eu^n \|_2^2 ) 
	+ 2 \B {\teu}{\imp{\hat{\breve{\bU}}}}{\behu} 
	+ 2 \B {\ext{\bu}}{\behu}{\behu} \\
	& + 2 \nu \| \nabla_h \behu \|_2^2 = - \langle \nabh e_p^n , \ehu^{n+1} + \eu^n \rangle_1 
	- 2 \gamma \langle \mAh \tephi \nabh \breve{\text{M}}^{n+\frac12} , 
	 \behu \rangle_1 \\ 
        &  \qquad \qquad \qquad  \qquad 
        - 2 \gamma \langle \mAh\ext{\phi}\nabh \emu ,  \behu \rangle_1 
        + 2 \langle \bm{\zeta}_1^n , \behu \rangle_1 .       
\end{aligned} 
  \label{convergence-rough-1} 
\end{equation} 
By the nonlinear identity~\eqref{B est-0-1} (in Lemma~\ref{B-est}), it is clear that 
\begin{equation} 
   \B {\ext{\bu}}{\behu}{\behu} = 0 . \label{convergence-rough-2-1} 
\end{equation} 
For the second term on the left hand side of~\eqref{convergence-rough-1}, we make use of inequality~\eqref{B est-0-2} and obtain 
\begin{equation} 
\begin{aligned} 
  & 
  2 \Big| \B {\teu}{\imp{\hat{\breve{\bU}}}}{\behu}  \Big| 
  \le \| \teu \|_2 ( \| \nabla_h \imp{\hat{\breve{\bU}}} \|_\infty \cdot \| \behu \|_2 
  + \| \imp{\hat{\breve{\bU}}} \|_\infty \cdot \| \nabla_h \behu \|_2 )  
\\
  & 
  \le C^* \| \teu \|_2 ( \| \behu \|_2 +  \| \nabla_h \behu \|_2 )  
  \le C^* \| \teu \|_2 \cdot (C_0 +1)  \| \nabla_h \behu \|_2  
\\
  & 
  \le \frac{(C^* (C_0 +1))^2}{2 \nu}  \| \teu \|_2^2 +  \frac{\nu}{2} \| \nabla_h \behu \|_2^2  
   \le \tilde{C}_2 ( 3 \| \eu^n \|_2^2 + \| \eu^{n-1} \|_2^2 ) 
   +  \frac{\nu}{2} \| \nabla_h \behu \|_2^2 , 
\end{aligned} 
  \label{convergence-rough-2-2} 
\end{equation} 
 with $\tilde{C}_2 = \frac{(C^* (C_0 +1))^2}{2 \nu}$, and the $W_h^{1, \infty}$ assumption~\eqref{assumption:W1-infty bound} for the constructed solution $\breve{\Phi}$ has been applied in the derivation. Notice that the discrete Pincar\'e inequality, $\| \behu \|_2 \le C_0  \| \nabla_h \behu \|_2$, (which comes from Proposition~\ref{prop: Poincare}), was used in the second step, due to the no-penetration boundary condition for $\behu$.  
 
In terms of numerical error inner product associated with the pressure gradient, we see that 
\begin{equation} 
    \langle \nabh e_p^n , \eu^k \rangle_1  = - \langle e_p^n , \nabla_h \cdot \eu^k \rangle_c 
    = 0 ,  \quad \mbox{since $\nabla_h \cdot \eu^n =0$} ,  \quad k = n, n+1 , 
    \label{convergence-rough-3-1} 
\end{equation} 
in which the summation by parts formula~\eqref{summation-2} has been applied. Regarding the other pressure gradient inner product term, we make use of~\eqref{error-pressure-correction} and obtain 
\begin{equation} 
\begin{aligned} 
  \langle \nabh e_p^n , \ehu^{n+1} \rangle_1  
  = & \langle \nabh e_p^n , \eu^{n+1} \rangle_1 + \frac12 \tau \langle \nabh e_p^n , 
  \nabla_h ( e_p^{n+1} - e_p^n ) \rangle_1    
\\
  = & 
  \frac12 \tau \langle \nabh e_p^n , 
  \nabla_h ( e_p^{n+1} - e_p^n ) \rangle_1    
\\ 
  = & 
   \frac14 \tau ( \| \nabh e_p^{n+1} \|_2^2 - \| \nabla_h e_p^n \|_2^2 
  - \| \nabla_h ( e_p^{n+1} - e_p^n ) \|_2^2 )  , 
\end{aligned} 
  \label{convergence-rough-3-2} 
\end{equation} 
in which the second step comes from the fact that $\langle \nabh e_p^n , \eu^{n+1} \rangle_1 =0$. For the second term on the right hand side of~\eqref{convergence-rough-1}, a direct application of discrete H\"older inequality implies that 
\begin{equation} 
\begin{aligned} 
  & 
  - 2 \gamma \langle \mAh \tephi \nabh \breve{\text{M}}^{n+\frac12} , 
	 \behu \rangle_1  \le 2 \gamma \| \tephi \|_2 
	 \cdot \| \nabh \breve{\text{M}}^{n+\frac12} \|_\infty  \cdot \| \behu \|_2 
\\
  & 
   \le 2 \gamma C^* \| \tephi \|_2  \cdot \| \behu \|_2  
  \le 2 \gamma C_0 C^* \| \tephi \|_2  \cdot \| \nabla_h \behu \|_2   
\\
  & 
  \le \frac{4 \gamma^2 C_0^2 ( C^* )^2}{\nu} \| \tephi \|_2^2  
  + \frac{\nu}{4} \| \nabla_h \behu \|_2^2  
  \le \tilde{C}_3 ( 3 \| \ephi^n \|_2^2 + \| \ephi^{n-1} \|_2^2 ) 
   +  \frac{\nu}{4} \| \nabla_h \behu \|_2^2 , 
\end{aligned} 
  \label{convergence-rough-4} 
\end{equation}  
with $\tilde{C}_3 = \frac{4 \gamma^2 C_0^2 ( C^* )^2}{\nu}$. Again, the $W_h^{1, \infty}$ assumption~\eqref{assumption:W1-infty bound mu} and the discrete Poincar\'e inequality, $\| \behu \|_2 \le C_0  \| \nabla_h \behu \|_2$, have been applied in the derivation. The local truncation error term could be controlled in a straightforward manner: 
\begin{equation} 
   2 \langle \bm{\zeta}_1^n , \behu \rangle_1 
   \le 2 \| \bm{\zeta}_1^n \|_2 \cdot \| \behu \|_2 
   \le 2 C_0 \| \bm{\zeta}_1^n \|_2 \cdot \| \nabla_h \behu \|_2  
   \le \frac{4 C_0^2}{\nu} \| \bm{\zeta}_1^n \|_2^2 
   + \frac{\nu}{4} \| \nabla_h \behu \|_2^2 .  \label{convergence-rough-5} 
\end{equation}  
Subsequently, a substitution of~\eqref{convergence-rough-2-1}-\eqref{convergence-rough-5} into \eqref{convergence-rough-1} yields 
\begin{equation} 
\begin{aligned} 
  & 
	\frac{1}{\tau} ( \| \ehu^{n+1} \|_2^2 - \| \eu^n \|_2^2 ) 
	+ \nu \| \nabla_h \behu \|_2^2 
	+ \frac14 \tau ( \| \nabh e_p^{n+1} \|_2^2 - \| \nabla_h e_p^n \|_2^2 ) 
\\	
	\le & 
        - 2 \gamma \langle \mAh\ext{\phi}\nabh \emu ,  \behu \rangle_1  
       +  \frac14 \tau \| \nabla_h ( e_p^{n+1} - e_p^n ) \|_2^2    
\\
    &   
    +  \tilde{C}_2 ( 3 \| \eu^n \|_2^2 + \| \eu^{n-1} \|_2^2 )     
    +  \tilde{C}_3 ( 3 \| \ephi^n \|_2^2 + \| \ephi^{n-1} \|_2^2 )  
  + \frac{4 C_0^2}{\nu} \| \bm{\zeta}_1^n \|_2^2   .   
\end{aligned} 
  \label{convergence-rough-6-1} 
\end{equation} 
Meanwhile, taking a discrete inner product with~\eqref{error-pressure-correction} by $2 \eu^{n+1}$ gives 
\begin{equation} 
  \| \eu^{n+1} \|_2^2 - \| \ehu^{n+1} \|_2^2 + \| \eu^{n+1} - \ehu^{n+1} \|_2^2 = 0 , \quad 
  \mbox{so that} \, \, \,  \| \eu^{n+1} \|_2^2 - \| \ehu^{n+1} \|_2^2 
  + \frac14 \tau^2 \| \nabla_h ( e_p^{n+1} - e_p^n ) \|_2^2 = 0 , 
  \label{convergence-rough-6-2} 
\end{equation} 
in which the divergence-free condition for $\eu^{n+1}$ has been used. In turn, a combination of~\eqref{convergence-rough-6-1} and \eqref{convergence-rough-6-2} results in 
\begin{equation} 
\begin{aligned} 
  & 
	\frac{1}{\tau} ( \| \eu^{n+1} \|_2^2 - \| \eu^n \|_2^2 ) 
	+ \nu \| \nabla_h \behu \|_2^2 
	+ \frac14 \tau ( \| \nabh e_p^{n+1} \|_2^2 - \| \nabla_h e_p^n \|_2^2 ) 
\\	
	\le & 
        - 2 \gamma \langle \mAh\ext{\phi}\nabh \emu ,  \behu \rangle_1          
\\
    &   
    +  \tilde{C}_2 ( 3 \| \eu^n \|_2^2 + \| \eu^{n-1} \|_2^2 )     
    +  \tilde{C}_3 ( 3 \| \ephi^n \|_2^2 + \| \ephi^{n-1} \|_2^2 )  
  + \frac{4 C_0^2}{\nu} \| \bm{\zeta}_1^n \|_2^2   .   
\end{aligned} 
  \label{convergence-rough-6-3} 
\end{equation} 

Now we proceed into a rough error estimate for the phase field error evolutionary equation. Taking a discrete inner product with~\eqref{error-CH} by $\emu$ leads to 
\begin{equation}
\begin{aligned}
\frac{1}{\tau} \langle\ephi^{n+1}, \emu \rangle_c & +\|\nabla_h \emu \|_2^2 
- \langle \mAh\ext{\phi} \nabla_h \emu , \behu \rangle_1 \\
&= \langle \mAh \tephi \nabla_h \emu , \imp{\hat{\breve{\bU}}} \rangle_1 
+ \langle \zeta_2^n , \emu \rangle_c + \frac{1}{\tau} \langle \ephi^n, \emu \rangle_c , 
\end{aligned}
\label{convergence-rough-7}
\end{equation}
with an application of summation by parts formula~\eqref{summation-5}. The right hand side terms could be analyzed as follows, with the help of the $\ell^\infty$ bound \eqref{assumption:W1-infty bound}: 
\begin{align} 
  &
  \langle \mAh \tephi \nabla_h \emu , \imp{\hat{\breve{\bU}}} \rangle_1  
  \le \| \tephi \|_2 \cdot \| \nabla_h \emu \|_2 \cdot \| \imp{\hat{\breve{\bU}}} \|_\infty 
  \le C^* \| \tephi \|_2 \cdot \| \nabla_h \emu \|_2  \nonumber 
\\
  \le & 
  (C^*)^2 \| \tephi \|_2^2  + \frac14 \| \nabla_h \emu \|_2^2 
  \le (C^*)^2 ( 3 \| \ephi^n \|_2^2  + \| \ephi^{n-1} \|_2^2 ) 
  + \frac14 \| \nabla_h \emu \|_2^2  , \label{convergence-rough-8-1} 
\\
  & 
  \langle \zeta_2^n , \emu \rangle_c  
   \le \| \zeta_2^n \|_{-1, h} \cdot \| \nabla_h \emu \|_2  
   \le C_1 \| \zeta_2^n \|_2 \cdot  \| \nabla_h \emu \|_2^2 
   \le C_1^2 \| \zeta_2^n \|_2^2 + \frac14 \| \nabla_h \emu \|_2^2  ,  
   \label{convergence-rough-8-2}  
\\
  & 
  \frac{1}{\tau} \langle \ephi^n, \emu \rangle_c 
  \le \frac{1}{\tau} \| \ephi^n \|_{-1, h} \cdot \| \nabla_h \emu \|_2 
  \le \frac{C_1}{\tau} \| \ephi^n \|_2 \cdot \| \nabla_h \emu \|_2  \nonumber 
\\
  &  \qquad \qquad \qquad 
  \le \frac{C_1^2}{\tau^2} \| \ephi^n \|_2^2 + \frac14 \| \nabla_h \emu \|_2 ^2 ,  
  \label{convergence-rough-8-3}  
\end{align} 
in which the preliminary estimate, $\| f \|_{-1, h} \le C_1 \| f \|_2$ for any $\overline{f} =0$ (as stated in Proposition~\ref{prop: Poincare}), has been repeatedly applied. In terms of the first term on the left hand side of~\eqref{convergence-rough-7}, we begin with the following expansion: 
\begin{equation} 
\begin{aligned} 
   \langle\ephi^{n+1}, \emu \rangle_c 
   = & \langle \ephi^{n+1} , F_{1 + \breve{\Phi}^n}(1 + \breve{\Phi}^{n+1}) - F_{1 + \phi^n}(1+\phi^{n+1}) \rangle_c  - \theta_0 \langle \ephi^{n+1} , \tephi \rangle_c 
\\
  & 
  + \langle \ephi^{n+1} , - F_{1 - \breve{\Phi}^n}(1 - \breve{\Phi}^{n+1}) + F_{1 - \phi^n}(1-\phi^{n+1}) \rangle_c  - \epsilon^2 \langle \ephi^{n+1} , \Dh \bbephi \rangle_c
\\
  & 
  + \tau \langle \ephi^{n+1} , \mathcal{N}(\breve{\Phi}^{n+1}) 
  -\mathcal{N}(\phi^{n+1})   \rangle_c  
  - \tau \langle \ephi^{n+1} , \mathcal{N}(\breve{\Phi}^n) - \mathcal{N}(\phi^n) \rangle_c . 
\end{aligned} 
  \label{convergence-rough-9-1}  
\end{equation} 
The first and third terms have been analyzed in~\eqref{nonlinear est-1-1} (given by Proposition~\ref{prop: nonlinear est-2}). The second and fourth terms could be bounded as follows:  
\begin{align} 
  - \theta_0 \langle \ephi^{n+1} , \tephi \rangle_c  
  = & - \theta_0 \langle \ephi^{n+1} , \frac32 \ephi^n - \frac12 \ephi^{n-1}  \rangle_c  
  \ge - \theta_0 \| \ephi^{n+1} \|_2 ( \frac32 \| \ephi^n \|_2 + \frac12 \| \ephi^{n-1} \|_2 ) , 
  \label{convergence-rough-9-2} 
\\
  - \langle \ephi^{n+1} , \Dh \bbephi \rangle_c 
  = & - \langle \ephi^{n+1} , \Dh ( \frac34 \ephi^{n+1} + \frac14 \ephi^{n-1} ) \rangle_c
  =  \langle \nabla_h \ephi^{n+1} , \nabla_h ( \frac34 \ephi^{n+1} 
   + \frac14 \ephi^{n-1} ) \rangle_c  \nonumber 
\\
  \ge & 
  \frac58  \| \nabla_h \ephi^{n+1} \|_2^2 - \frac18 \| \nabla_h \ephi^{n-1} \|_2^2 . 
  \label{convergence-rough-9-3} 
\end{align}  
It is observed that the fifth term in the expansion~\eqref{convergence-rough-9-1} must be non-negative, due to the monotone property of ${\mathcal N}$, as well as the fact that $\ephi^{n+1} = \breve{\Phi}^{n+1} - \phi^{n+1}$: 
\begin{equation}  
  \langle \ephi^{n+1} , \mathcal{N}(\breve{\Phi}^{n+1}) 
  -\mathcal{N}(\phi^{n+1})   \rangle_c  
  = \langle \breve{\Phi}^{n+1} - \phi^{n+1} , \mathcal{N}(\breve{\Phi}^{n+1}) 
  -\mathcal{N}(\phi^{n+1})   \rangle_c  \ge 0 . 
   \label{convergence-rough-9-4} 
\end{equation}  
Regarding the last term in the expansion~\eqref{convergence-rough-9-1}, we see that both the constructed solution $\breve{\Phi}^n$ and the numerical solution $\phi^n$ preserve the phase separation property~\eqref{separation-2}. In turn, an application of the preliminary error estimate~\eqref{nonlinear est-0-1} (in Proposition~\ref{prop: nonlinear est-1}) reveals that  
\begin{equation} 
\begin{aligned} 
  & 
  \| {\mathcal N} ( \breve{\Phi}^n ) - {\mathcal N} ( \phi^n) \|_2 \le C_\delta \| \ephi^n \|_2 ,  \quad 
  \mbox{so that} 
\\
  & 
  - \langle \ephi^{n+1} , \mathcal{N}(\breve{\Phi}^n) - \mathcal{N}(\phi^n) \rangle_c 
  \ge - C_\delta \| \ephi^{n+1} \|_2 \cdot \| \ephi^n \|_2 . 
\end{aligned} 
  \label{convergence-rough-9-5} 
\end{equation} 
Subsequently, a substitution of~\eqref{convergence-rough-8-1}-\eqref{convergence-rough-9-5} into \eqref{convergence-rough-7} yields 
\begin{equation}
\begin{aligned}
  & 
 \frac{\epsilon^2}{\tau} ( \frac58  \| \nabla_h \ephi^{n+1} \|_2^2 
 - \frac18 \| \nabla_h \ephi^{n-1} \|_2^2 )  + \frac14 \|\nabla_h \emu \|_2^2 
- \langle \mAh\ext{\phi} \nabla_h \emu , \behu \rangle_1 \\
 \le & 
 (C^*)^2 ( 3 \| \ephi^n \|_2^2  + \| \ephi^{n-1} \|_2^2 ) 
  + C_1^2 \| \zeta_2^n \|_2^2  +  \frac{C_1^2}{\tau^2} \| \ephi^n \|_2^2  
\\
  & 
  + ( 2 C_\delta + \frac32 \theta_0 + 1) \frac{1}{\tau} \| \ephi^{n+1} \|_2 \cdot \| \ephi^n \|_2 
   + \frac{\theta_0}{2 \tau} \| \ephi^{n+1} \|_2 \cdot \| \ephi^{n-1} \|_2 ,  \quad 
   \mbox{if $C_\delta \tau \le 1$} .  
\end{aligned}
\label{convergence-rough-10}
\end{equation}

A combination of~\eqref{convergence-rough-6-3} and \eqref{convergence-rough-10} gives 
\begin{equation}
\begin{aligned} 
  & 
  \frac{1}{2 \gamma} \| \eu^{n+1} \|_2^2  
  +  \frac{5 \epsilon^2}{8}  \| \nabla_h \ephi^{n+1} \|_2^2 
  + \frac{\tau^2}{8 \gamma} \| \nabla_h e_p^{n+1} \|_2^2 
\\
  \le & \frac{1}{2 \gamma} \| \eu^n \|_2^2 
   + \frac{\epsilon^2}{8} \| \nabla_h \ephi^{n-1} \|_2^2   
   + \frac{\tau^2}{8 \gamma} \| \nabla_h e_p^n \|_2^2    
   + C_1^2 \tau \| \zeta_2^n \|_2^2  
   + \frac{2 C_0^2 \tau}{\gamma \nu} \| \bm{\zeta}_1^n \|_2^2  
   \\
    & 
  + ( 3 (C^*)^2 \tau + \frac{3 \tilde{C}_3}{2 \gamma} \tau + \frac{C_1^2}{\tau} ) \| \ephi^n \|_2^2  
 + ( ( C^* )^2 + \frac{\tilde{C}_3}{2 \gamma} ) \tau \| \ephi^{n-1} \|_2^2  
\\
  & 
  +  \frac{\tilde{C}_2 \tau}{2 \gamma} ( 3 \| \eu^n \|_2^2 + \| \eu^{n-1} \|_2^2 )   
  + \tilde{C}_4 \| \ephi^{n+1} \|_2 \cdot \| \ephi^n \|_2 
   + \frac12 \theta_0 \| \ephi^{n+1} \|_2 \cdot \| \ephi^{n-1} \|_2  , 
\end{aligned}
  \label{convergence-rough-11-1}
\end{equation}
with $\tilde{C}_4 = 2 C_\delta + \frac32 \theta_0 + 1$. Notice that the term $\langle \mAh\ext{\phi}\nabh \emu ,  \behu \rangle_1$ cancels each other between \eqref{convergence-rough-6-3} and \eqref{convergence-rough-10}, and this fact plays a crucial role in the error estimate. For the right hand side of \eqref{convergence-rough-11-1}, the following estimates are available, which come from the a-priori assumption \eqref{a priori-1}: 
\begin{equation}
\begin{aligned} 
  & 
( \frac{1}{2 \gamma} + \frac{3 \tilde{C}_2 \tau}{2 \gamma} ) \| \eu^n \|_2^2 
\le \frac{1}{\gamma} \| \eu^n \|_2^2  \le C ( \tau^\frac{11}{2} + h^\frac{15}{2} ) ,  \quad 
\frac{\tilde{C}_2 \tau}{2 \gamma} \| \eu^{n-1} \|_2^2 \le C \tau ( \tau^\frac{11}{2} + h^\frac{15}{2} ) , 
\\
 &
 \frac{\epsilon^2}{8} \| \nabla_h \ephi^{n-1} \|_2^2   \le \frac{\epsilon^2}{4} ( \tau^\frac{11}{2} + h^\frac{15}{2} ) , \quad 
 ( ( C^* )^2 + \frac{\tilde{C}_3}{2 \gamma} ) \tau \| \ephi^{n-1} \|_2^2   
 \le C \tau ( \tau^\frac{11}{2} + h^\frac{15}{2} ) , 
\\
  & 
  ( 3 (C^*)^2 \tau + \frac{3 \tilde{C}_3}{2 \gamma} \tau + \frac{C_1^2}{\tau} ) \| \ephi^n \|_2^2  
  \le ( \frac{C_1^2}{\tau} + 1 ) \| \ephi^n \|_2^2  \le C ( \tau^\frac{9}{2} + h^\frac{13}{2} ) , 
\\
  & 
  C_1^2 \tau \| \zeta_2^n \|_2^2  , \, \, 
  \frac{2 C_0^2 \tau}{\gamma \nu} \| \bm{\zeta}_1^n \|_2^2  
  \le C \tau ( \tau^6 + h^8 ) ,  \quad 
  \frac{\tau^2}{8 \gamma} \| \nabla_h e_p^n \|_2^2  \le C ( \tau^\frac{11}{2} + h^\frac{15}{2} ) , 
\\
  & 
  \tilde{C}_4 \| \ephi^{n+1} \|_2 \cdot \| \ephi^n \|_2 
  \le C_1 \tilde{C}_4 \| \nabla_h \ephi^{n+1} \|_2 \cdot \| \ephi^n \|_2
  \le 2 C_1^2 \tilde{C}_4^2 \epsilon^{-2}  \| \ephi^n \|_2^2 
  + \frac{\epsilon^2}{8} \| \nabla_h \ephi^{n+1} \|_2^2 , 
\\
  & 
  \frac12 \theta_0 \| \ephi^{n+1} \|_2 \cdot \| \ephi^{n-1} \|_2 
  \le \frac12 C_1^2 \theta_0^2 \epsilon^{-2}  \| \ephi^{n-1} \|_2^2 
  + \frac{\epsilon^2}{8} \| \nabla_h \ephi^{n+1} \|_2^2 ,   
\end{aligned} 
  \label{convergence-rough-11-2}
\end{equation}
where the fact that $\| f \|_{-1,h} \leq C_1 \| f \|_2$, as well as the refinement constraint $C_1 h \le \tau \le C_2 h$, have been repeatedly used. Going back~\eqref{convergence-rough-11-1}, we obtain 
\begin{equation} 
\begin{aligned} 
  & 
  \frac{1}{2 \gamma} \| \eu^{n+1} \|_2^2  
  +  \frac{3 \epsilon^2}{8}  \| \nabla_h \ephi^{n+1} \|_2^2 
  \le   C ( \tau^\frac{9}{2} + h^\frac{13}{2} ) 
  + C \epsilon^{-2} ( \| \ephi^n \|_2^2  + \| \ephi^{n-1} \|_2^2 ) 
  \le C ( \tau^\frac{9}{2} + h^\frac{13}{2} ) ,  
\\
  & \mbox{so that} \quad 
  \| \ephi^{n+1} \|_2 + \| \nabla_h \ephi^{n+1} \|_2 \le (C_0 + 1) \| \nabla_h \ephi^{n+1} \|_2 
  \le C ( \tau^\frac{9}{4} + h^\frac{13}{4} ) , 
\end{aligned} 
    \label{convergence-rough-11-3}
\end{equation}
under the linear refinement requirement $C_1 h \le \tau \le C_2 h$, provided that $\tau$ and $h$ are sufficiently small. As a direct consequence of the rough error estimate \eqref{convergence-rough-11-3}, an application of inverse inequality reveals that 
\begin{align}
&
\| \ephi^{n+1} \|_\infty  \leq \frac{C \| \ephi^{n+1} \|_{H_h^1} }{h^\frac{1}{2}}
\le \hat {C}_1 ( \tau^\frac74 + h^\frac{11}{4} )  \le \tau + h \le \frac{\delta}{2} , 
   \label{convergence-rough-11-4}  
\\
&
  \| \nabla_h \ephi^{n+1} \|_\infty  \leq \frac{C \| \ephi^{n+1} \|_\infty }{h}
\le C \hat {C}_1 ( \tau^\frac34 + h^\frac{7}{4} )  \le \tau^\frac12 + h \le 1 . 
\label{convergence-rough-11-5}  
\end{align}
With the help of the separation estimate~\eqref{separation-1} for the constructed profile $\breve{\Phi}$, the phase separation property becomes available for $\phi^{n+1}$, at the next time step: 
	\begin{equation}  \label{separation-4} 
		-1+ \delta \le \phi^{n+1} \le 1- \delta .
	\end{equation} 
In addition, a $W_h^{1, \infty}$ bound of $\phi^{n+1}$ is also valid: 
\begin{equation} 
  \| \nabla_h \phi^{n+1} \|_\infty \le \| \nabla_h \breve{\Phi}^{n+1} \|_\infty 
  + \| \nabla_h \ephi^{n+1} \|_\infty \le C^* +1 = \tilde{C}_1 ,  
  \label{a priori-4} 
\end{equation} 
in combination with the $W_h^{1, \infty}$ assumption~\eqref{assumption:W1-infty bound} for the constructed solution $\breve{\Phi}$.

\subsection{The refined error estimate} 

Before proceeding into the refined error estimate, we derive a few preliminary results for the nonlinear error term. The following nonlinear error is introduced, with $\emu = {\mathcal NLE}^n - \epsilon^2 \Dh\bbephi$: 
\begin{equation} 
\begin{aligned} 
  {\mathcal NLE}^n = & F_{1+\breve{\Phi}^n}(1+\breve{\Phi}^{n+1}) - F_{1+\phi^n}(1+\phi^{n+1}) 
	 - F_{1-\breve{\Phi}^n}(1-\breve{\Phi}^{n+1}) + F_{1-\phi^n}(1-\phi^{n+1}) 
\\
	 &
	  -\theta_0\tephi +\tau \Big(\mathcal{N}(\breve{\Phi}^{n+1}) -\mathcal{N}(\phi^{n+1}) -\mathcal{N}(\breve{\Phi}^n) +\mathcal{N}(\phi^n) \Big) .  
\end{aligned} 
  \label{nonlinear est-2} 
\end{equation} 
We see that both the constructed solution $\breve{\Phi}^k$ and numerical solution $\phi^k$ preserve the phase separation property~\eqref{separation-2}, for $k= n, n+1$, as given by~\eqref{separation-1}, \eqref{separation-3} and \eqref{separation-4}. In addition, the additional conditions~\eqref{nonlinear est-condition-1} are satisfied by $\breve{\Phi}^k$ and $\phi^k$, as given by \eqref{assumption:W1-infty bound}, \eqref{a priori-2}, \eqref{a priori-3-2}, \eqref{convergence-rough-11-4} and \eqref{a priori-4}, respectively. Therefore, an application of the preliminary error estimate~\eqref{nonlinear est-0-2} (stated in Proposition~\ref{prop: nonlinear est-1}) indicates that   
\begin{equation} 
\begin{aligned} 
  & 
   \| \nabla_h ( {\mathcal N} ( \breve{\Phi}^n ) - {\mathcal N} ( \phi^n) ) \|_2 
  \le C_\delta ( \| \ephi^n \|_2 + \| \nabla_h \ephi^n \|_2 ) ,  
\\
  &  
  \| \nabla_h ( {\mathcal N} ( \breve{\Phi}^{n+1} ) - {\mathcal N} ( \phi^{n+1}) ) \|_2 
  \le C_\delta ( \| \ephi^{n+1} \|_2 + \| \nabla_h \ephi^{n+1} \|_2 ) .  
\end{aligned} 
  \label{nonlinear est-3} 
\end{equation} 
Similarly, the additional conditions in~\eqref{nonlinear est-condition-2} are satisfied, so that we are able to apply Proposition~\ref{prop: nonlinear est-2} and obtain~\eqref{nonlinear est-1-2}. A combination of~\eqref{nonlinear est-1-2} with \eqref{nonlinear est-3} gives 
\begin{equation} 
\begin{aligned} 
  \| \nabla_h {\mathcal NLE}^n \|_2  
  \le & C_\delta ( \| \nabla_h \ephi^n \|_2 + \| \nabla_h \ephi^{n+1} \|_2 ) 
  +  \theta_0 \| \nabla_h \tephi \|_2  
\\
  \le & 
  C_\delta ( \| \nabla_h \ephi^n \|_2 + \| \nabla_h \ephi^{n+1} \|_2 ) 
  +  \theta_0 ( \frac32 \| \nabla_h \ephi^n \|_2 + \frac12 \| \nabla_h \ephi^{n-1} \|_2 ) ,  
\end{aligned} 
  \label{nonlinear est-4} 
\end{equation} 
in which the discrete Poincar\'e inequality (stated in Proposition~\ref{prop: Poincare}) has been applied. 

The rough error estimate~\eqref{convergence-rough-6-3} is still valid, for the velocity error evolutionary equations. Now we carry out a refined rough error estimate for the phase variable error evolutionary equation. Taking a discrete inner product with~\eqref{error-CH} by $-\Delta_h \bbephi$ gives  
\begin{equation}
\begin{aligned}
- \frac{1}{\tau} \langle\ephi^{n+1} - \ephi^n, \Delta_h \bbephi \rangle_c & 
- \langle \nabla_h \emu ,  \nabla_h \Delta_h \bbephi \rangle_c 
+ \langle \mAh\ext{\phi} \nabla_h \Delta_h \bbephi , \behu \rangle_1 \\
&= - \langle \mAh \tephi \nabla_h \Delta_h \bbephi , \imp{\hat{\breve{\bU}}} \rangle_1 
- \langle \zeta_2^n , \Delta_h \bbephi \rangle_c  .  
\end{aligned}
  \label{convergence-1} 
\end{equation}
The temporal discretization term could be analyzed as follows: 
\begin{equation} 
\begin{aligned} 
  & 
  - \langle \ephi^{n+1} - \ephi^n, \Delta_h \bbephi \rangle_c 
  = - \langle \ephi^{n+1} - \ephi^n, \Delta_h (\frac34 \ephi^{n+1} 
   + \frac14 \ephi^{n-1} ) \rangle_c  
\\
  = & 
    \langle \nabla_h (\ephi^{n+1} - \ephi^n ) , \nabla_h (\frac34 \ephi^{n+1} 
   + \frac14 \ephi^{n-1} ) \rangle_c  
\\
  = & 
    \frac12 \langle \nabla_h (\ephi^{n+1} - \ephi^n ) , 
     \nabla_h ( \ephi^{n+1} + \ephi^n ) \rangle_c 
   + \frac14 
    \langle \nabla_h (\ephi^{n+1} - \ephi^n ) , \nabla_h ( \ephi^{n+1} 
   - 2 \ephi^n +  \ephi^{n-1} ) \rangle_c   
\\
  \ge & 
    \frac12 ( \| \nabla_h \ephi^{n+1} \|_2^2 - \| \nabla_h \ephi^n \|_2^2 )  
    + \frac18 ( \| \nabla_h (\ephi^{n+1} - \ephi^n ) \|_2^2  
    - \| \nabla_h ( \ephi^n -  \ephi^{n-1} ) \|_2^2 )  .   
\end{aligned} 
  \label{convergence-2} 
\end{equation} 
The estimate for the diffusion part is straightforward:  
\begin{align} 
  - \langle \nabla_h \emu ,  \nabla_h \Delta_h \bbephi \rangle_c  
  = & - \langle \nabla_h {\mathcal NLE}^n - \epsilon^2 \nabla_h \Delta_h \bbephi ,  
   \nabla_h \Delta_h \bbephi \rangle_c    \nonumber 
\\
  = & 
    \epsilon^2 \| \nabla_h \Delta_h \bbephi \|_2^2 
  - \langle \nabla_h {\mathcal NLE}^n ,  \nabla_h \Delta_h \bbephi \rangle_c  
  \nonumber 
\\
  \ge & 
  \frac{\epsilon^2}{2} \| \nabla_h \Delta_h \bbephi \|_2^2 
  - \frac12 \epsilon^{-2} \| \nabla_h {\mathcal NLE}^n \|_2^2 ,  \label{convergence-3-1}  
\\
 \| \nabla_h \Delta_h \bbephi \|_2^2  = & 
 \| \nabla_h \Delta_h ( \frac34 \ephi^{n+1} + \frac14 \ephi^{n-1} ) \|_2^2 
 \ge \frac38 \| \nabla_h \Delta_h \ephi^{n+1} \|_2^2 
  - \frac18 \| \nabla_h \Delta_h \ephi^{n-1} \|_2^2 , 
   \label{convergence-3-2}  
\end{align} 
in which the Cauchy inequality has been repeatedly applied in the derivation. Meanwhile, the analysis for the right hand side terms of~\eqref{convergence-1} could be performed in a similar  fashion as in \eqref{convergence-rough-8-1}-\eqref{convergence-rough-8-2}:  
\begin{align} 
  &
  - \langle \mAh \tephi \nabla_h \Delta_h \bbephi , \imp{\hat{\breve{\bU}}} \rangle_1  
  \le \| \tephi \|_2 \cdot \| \nabla_h \Delta_h \bbephi \|_2 \cdot \| \imp{\hat{\breve{\bU}}} \|_\infty  
  \nonumber 
\\
  \le & 
  C^* \| \tephi \|_2 \cdot \| \nabla_h \Delta_h \bbephi \|_2  
  \le  
  2 (C^*)^2 \epsilon^{-2} \| \tephi \|_2^2  
   + \frac{\epsilon^2}{8} \| \nabla_h \Delta_h \bbephi \|_2^2   \nonumber 
\\
  \le & 
  2  (C^*)^2 \epsilon^{-2} ( 3 \| \ephi^n \|_2^2  + \| \ephi^{n-1} \|_2^2 ) 
  + \frac{\epsilon^2}{8} \| \nabla_h \Delta_h \bbephi \|_2^2  , \label{convergence-4-1} 
\\
  & 
  - \langle \zeta_2^n , \Delta_h \bbephi \rangle_c  
   \le \| \zeta_2^n \|_{-1, h} \cdot \| \nabla_h \Delta_h \bbephi \|_2  
   \le C_1 \| \zeta_2^n \|_2 \cdot  \| \nabla_h \Delta_h \bbephi \|_2^2  \nonumber 
\\
  \le & 
    2 C_1^2 \epsilon^{-2} \| \zeta_2^n \|_2^2 
   + \frac{\epsilon^2}{8} \| \nabla_h \Delta_h \bbephi \|_2^2   .   
   \label{convergence-4-2}  
\end{align} 
Subsequently, a substitution of~\eqref{convergence-2}-\eqref{convergence-4-2} into \eqref{convergence-1} yields 
\begin{equation}
\begin{aligned} 
  & 
  \frac{1}{2 \tau} ( \| \nabla_h \ephi^{n+1} \|_2^2 - \| \nabla_h \ephi^n \|_2^2 )  
    + \frac{1}{8 \tau} ( \| \nabla_h (\ephi^{n+1} - \ephi^n ) \|_2^2  
    - \| \nabla_h ( \ephi^n -  \ephi^{n-1} ) \|_2^2 )  
\\
  & 
    + \frac{3 \epsilon^2}{32} \| \nabla_h \Delta_h \ephi^{n+1} \|_2^2  
  \le  
   \frac12 \epsilon^{-2} \| \nabla_h {\mathcal NLE}^n \|_2^2     
- \langle \mAh\ext{\phi} \nabla_h \Delta_h \bbephi , \behu \rangle_1 
\\
   &  \qquad \qquad \qquad \qquad 
   + 2  (C^*)^2 \epsilon^{-2} ( 3 \| \ephi^n \|_2^2  + \| \ephi^{n-1} \|_2^2 )  
     + 2 C_1^2 \epsilon^{-2} \| \zeta_2^n \|_2^2  
     + \frac{\epsilon^2}{32} \| \nabla_h \Delta_h \ephi^{n-1} \|_2^2   .  
\end{aligned}
  \label{convergence-5} 
\end{equation}

Meanwhile, in the preliminary estimate~\eqref{convergence-rough-6-3}, the nonlinear inner product term could be expanded as follows: 
\begin{align} 
  - \langle \mAh\ext{\phi}\nabh \emu ,  \behu \rangle_1  
  = & - \langle \mAh\ext{\phi}\nabh {\mathcal NLE}^n ,  \behu \rangle_1   \nonumber 
\\
  & 
    + \epsilon^2 \langle \mAh\ext{\phi}\nabh \Delta_h \bbephi ,  \behu \rangle_1 ,  
    \label{convergence-6-1} 
\\  
  - \langle \mAh\ext{\phi}\nabh {\mathcal NLE}^n ,  \behu \rangle_1  
  \le & \| \ext{\phi} \|_\infty \cdot \| \nabh {\mathcal NLE}^n \|_2 \cdot \|  \behu \|_2  
  \nonumber 
\\
  \le & 
  2 C_0 \| \nabh {\mathcal NLE}^n \|_2 \cdot \| \nabla_h \behu \|_2   \nonumber 
\\ 
  \le & 
     \frac{4 \gamma C_0^2}{\nu} \| \nabh {\mathcal NLE}^n \|_2 ^2 
  + \frac{\nu}{4 \gamma} \| \nabla_h \behu \|_2^2 . \label{convergence-6-2} 
\end{align} 
Its substitution into \eqref{convergence-rough-6-3} gives 
\begin{equation} 
\begin{aligned} 
  & 
	\frac{1}{\tau} ( \| \eu^{n+1} \|_2^2 - \| \eu^n \|_2^2 ) 
	+ \frac{\nu}{2} \| \nabla_h \behu \|_2^2 
	+ \frac14 \tau ( \| \nabh e_p^{n+1} \|_2^2 - \| \nabla_h e_p^n \|_2^2 )       
\\
    \le &   
       \tilde{C}_2 ( 3 \| \eu^n \|_2^2 + \| \eu^{n-1} \|_2^2 )     
    +  \tilde{C}_3 ( 3 \| \ephi^n \|_2^2 + \| \ephi^{n-1} \|_2^2 )  
  + \frac{4 C_0^2}{\nu} \| \bm{\zeta}_1^n \|_2^2  
\\
  & 
  + \frac{8 \gamma^2 C_0^2}{\nu} \| \nabh {\mathcal NLE}^n \|_2 ^2  
  + 2 \gamma \epsilon^2 \langle \mAh\ext{\phi}\nabh \Delta_h \bbephi , 
   \behu \rangle_1  .   
\end{aligned} 
  \label{convergence-6-3} 
\end{equation} 

Finally, a combination of~\eqref{convergence-5} and \eqref{convergence-6-3} leads to 
\begin{equation}
\begin{aligned} 
  & 
  \frac{1}{2 \tau} ( \| \nabla_h \ephi^{n+1} \|_2^2 - \| \nabla_h \ephi^n \|_2^2 )  
    + \frac{1}{8 \tau} ( \| \nabla_h (\ephi^{n+1} - \ephi^n ) \|_2^2  
    - \| \nabla_h ( \ephi^n -  \ephi^{n-1} ) \|_2^2 )  
\\
  & 
  + \frac{\epsilon^{-2}}{2 \gamma \tau} ( \| \eu^{n+1} \|_2^2 - \| \eu^n \|_2^2 ) 
	+ \frac{\epsilon^{-2}}{8 \gamma} \tau ( \| \nabh e_p^{n+1} \|_2^2 - \| \nabla_h e_p^n \|_2^2 )     
    + \frac{3 \epsilon^2}{32} \| \nabla_h \Delta_h \ephi^{n+1} \|_2^2  
\\
  \le & 
   \Big( \frac{4 \gamma C_0^2}{\nu} + \frac12 \Big) \epsilon^{-2} 
   \| \nabla_h {\mathcal NLE}^n \|_2^2    
   + \frac{\tilde{C}_2 \epsilon^{-2}}{2 \gamma} ( 3 \| \eu^n \|_2^2 + \| \eu^{n-1} \|_2^2 )     
  + \frac{2 C_0^2 \epsilon^{-2}}{\gamma \nu} \| \bm{\zeta}_1^n \|_2^2      
\\
   & 
   + \Big( \frac{\tilde{C}_3}{2 \gamma} + 2  (C^*)^2 \Big) \epsilon^{-2} 
   ( 3 \| \ephi^n \|_2^2  + \| \ephi^{n-1} \|_2^2 )  
     + 2 C_1^2 \epsilon^{-2} \| \zeta_2^n \|_2^2  
     + \frac{\epsilon^2}{32} \| \nabla_h \Delta_h \ephi^{n-1} \|_2^2   .  
\end{aligned}
  \label{convergence-7-1} 
\end{equation}
Similarly, the term $\langle \mAh\ext{\phi}\nabh  \Delta_h \bbephi ,  \behu \rangle_1$ cancels each other between \eqref{convergence-5} and \eqref{convergence-6-3}, and such a cancellation makes the error estimate go through. Moreover, by the preliminary nonlinear estimate~\eqref{nonlinear est-4}, we arrive at 
\begin{equation}
\begin{aligned} 
  & 
  \frac{1}{2 \tau} ( \| \nabla_h \ephi^{n+1} \|_2^2 - \| \nabla_h \ephi^n \|_2^2 )  
    + \frac{1}{8 \tau} ( \| \nabla_h (\ephi^{n+1} - \ephi^n ) \|_2^2  
    - \| \nabla_h ( \ephi^n -  \ephi^{n-1} ) \|_2^2 )  
\\
  & 
  + \frac{\epsilon^{-2}}{2 \gamma \tau} ( \| \eu^{n+1} \|_2^2 - \| \eu^n \|_2^2 ) 
	+ \frac{\epsilon^{-2}}{8 \gamma} \tau ( \| \nabh e_p^{n+1} \|_2^2 - \| \nabla_h e_p^n \|_2^2 )     
    + \frac{3 \epsilon^2}{32} \| \nabla_h \Delta_h \ephi^{n+1} \|_2^2  
\\
  \le & 
  \tilde{C}_5 \| \nabla_h \ephi^{n+1} \|_2^2 + \tilde{C}_6 \| \nabla_h \ephi^n \|_2^2 
  + \tilde{C}_7 \| \nabla_h \ephi^{n-1} \|_2^2     
   + \frac{\tilde{C}_2 \epsilon^{-2}}{2 \gamma} ( 3 \| \eu^n \|_2^2 + \| \eu^{n-1} \|_2^2 )    
\\
  &  
  + \frac{2 C_0^2 \epsilon^{-2}}{\gamma \nu} \| \bm{\zeta}_1^n \|_2^2      
     + 2 C_1^2 \epsilon^{-2} \| \zeta_2^n \|_2^2  
     + \frac{\epsilon^2}{32} \| \nabla_h \Delta_h \ephi^{n-1} \|_2^2  ,  
\end{aligned}
  \label{convergence-7-2} 
\end{equation}
with $\tilde{C}_5 = ( \frac{4 \gamma C_0^2}{\nu} + \frac12 ) C_\delta \epsilon^{-2}$, 
$\tilde{C}_6 = \tilde{C}_5 + \Big( 6 \theta_0^2 + 3 ( \frac{\tilde{C}_3}{2 \gamma} + 2  (C^*)^2 ) C_0^2 \Big) \epsilon^{-2}$, $\tilde{C}_7 = \Big(2 \theta_0^2 + ( \frac{\tilde{C}_3}{2 \gamma} + 2  (C^*)^2 ) C_0^2 \Big) \epsilon^{-2}$. Notice that the discrete Poincar\'e inequality (stated in Proposition~\ref{prop: Poincare}) has been repeatedly applied. Therefore, with sufficiently small $\tau$ and $h$, an application of discrete Gronwall inequality leads to the desired higher order error estimate 
\begin{equation}
  \| \nabla_h \ephi^{n+1} \|_2 + \| \eu^{n+1} \|_2 + \tau \| \nabla_h e_p^{n+1} \|_2 
  + \Bigl(  \frac{\epsilon^2}{8} \tau   \sum_{k=1}^{n+1} 
    \| \nabla_h \Delta_h \ephi^k \|_2^2 
   \Bigr)^\frac12  \le C ( \tau^3 + h^4 ) , 
	\label{convergence-8}
\end{equation} 
based on the higher order truncation error accuracy, $\| \bm{\zeta}_1^n \|_2$, $\| \zeta_2^n \|_2 \le C (\tau^3 + h^4)$. As a result, a refined error estimate is obtained.  

With the higher order convergence estimate \eqref{convergence-8} in hand, the a-priori assumption in~\eqref{a priori-1} is recovered at the next time step $t^{n+1}$:  
\begin{equation} 
\begin{aligned} 
  & 
  \| \eu^{n+1} \|_2 , \, \| \nabla_h \ephi^{n+1} \|_2 
  \le C (\tau^3 + h^4 ) \le \tau^\frac{11}{4} + h^\frac{15}{4},   
\\
  &  
  \| \nabla_h e_p^{n+1} \|_2 \le C \tau^{-1} ( \tau^3 + h^4 ) \le C ( \tau^2 + h^3) 
  \le \tau^\frac74 + h^\frac{11}{4} , 
\end{aligned} 
	\label{a priori-5}  
\end{equation} 
provided that $\tau$ and $h$ are sufficiently small. As a result, an induction analysis could be effectively applied, and the higher order convergence analysis is finished.

Meanwhile, the following identity is observed:   
\begin{equation} 
  \tilde{\phi}^n = e_\phi^n - ( \tau^2 \Phi_{\tau, 1} + h^2 \Phi_{h,1}) , \quad 
  \tilde{\bu}^n = \eu^n - (\tau^2 \bU_{\tau, 1} + h^2 \bU_{h,1}) , \quad 
  \tilde{p}^n = e_p^n - (\tau^2 P_{\tau, 1} + h^2 P_{h,1}) , 
  \label{convergence-9-1} 
\end{equation} 
which comes from a comparison between the error functions defined in~\eqref{error function-1}-\eqref{error function-2} and \eqref{error function-3}, as well as the expansion~\eqref{construction-1}, \eqref{construction-2} for the constructed profiles. As a result, the original error estimate~\eqref{convergence-0} is a direct consequence of the $O (\tau^3 + h^4)$ estimate~\eqref{convergence-8} for $( e_\phi^n , \eu^n, e_p^n)$, combined with the numerical error expansion~\eqref{convergence-9-1}. This completes the proof of Theorem~\ref{thm: convergence}.   

\begin{rem}
	The proposed numerical system \eqref{CN-NS}-\eqref{CN-incompressible} depends on $\phi^{n+1}$ in a highly nonlinear and singular way, because of the same nature of the Flory-Huggins free energy. The positivity-preserving property, stated in Theorem~\ref{thm: positivity}, only ensures the convexity of the nonlinear approximation in the logarithmic parts. Such a rough knowledge enables one to derive a rough error estimate, by making use of a higher order consistency analysis. However, it is noticed that the accuracy order in \eqref{convergence-rough-11-3} is at least half order lower than the a-priori estimate \eqref{a priori-1}, as well as the lower rate of the $W_h^{1, \infty}$ errors in \eqref{convergence-rough-11-4}-\eqref{convergence-rough-11-5}, which comes from an application of the inverse inequality. Of course, the a-priori assumption could not be recovered by the lower accuracy rate in \eqref{convergence-rough-11-3}.  Instead, the separation properties \eqref{separation-3}, \eqref{separation-4} are used to derive a much sharper estimate~\eqref{nonlinear est-4} for the logarithmic gradient error term, and this sharper inequality leads to a refined error estimate~\eqref{convergence-8}, which keeps the higher accuracy order in the consistency analysis. 
\end{rem}

\begin{rem} 
The error estimate for the pressure variable has not been reported in~\eqref{convergence-0}. In fact, based on the higher order convergence estimate \eqref{convergence-8}, combined with the numerical error expansion~\eqref{convergence-9-1}, we are able to derive an optimal rate error estimate for the pressure variable, in the $\ell^\infty (0, T; H_h^1)$ norm, for any $n \ge 0$:  
\begin{equation} 
\begin{aligned} 
  & 
  \| \nabla_h e_p^n \|_2 \le C \tau^{-1} ( \tau^3 + h^4 ) \le C ( \tau^2 + h^3)  ,  
\\
  & 
  \| \nabla_h ( e_p^n - \tilde{p}^n ) \|_2 
  = \| \nabla_h ( \tau^2 P_{\tau, 1} + h^2 P_{h,1} ) \|_2 \le C ( \tau^2 + h^2 ) ,   \quad 
  \mbox{so that} 
\\
  & 
   \| \nabla_h \tilde{p}^n \|_2  
   \le \| \nabla_h e_p^n \|_2 + \| \nabla_h ( e_p^n - \tilde{p}^n ) \|_2  
   \le C (  \tau^2 + h^2 ) , 
\end{aligned} 
  \label{convergence-p-1} 
\end{equation} 
under the linear refinement requirement $C_1 h \le \tau \le C_2 h$. The $H^1$ bound of the constructed functions, $P_{\tau, 1}$ and $P_{h,1}$, has also been applied in the derivation.    
\end{rem}

\begin{rem} 
The accuracy order and convergence rate of the proposed numerical scheme~\eqref{CN-NS}-\eqref{CN-incompressible} have been demonstrated by a few numerical experiments in the recent work~\cite{chen24a}, which validates the convergence and error estimate of this article. 
In more details, a few analytic functions are chosen to be the exact solutions:  
\begin{equation} 
\begin{aligned}
	&\phi_e (x,y,t) = 0.5\sin(2\pi x)\cos(2\pi y)\cos t + 0.1,  \quad 
	p_e (x,y,t)  = \sin t \sin(2\pi x) ,  \\
	&u_e (x,y,t) = -\cos t \cos(2\pi x)\sin(2\pi y ) , \quad 
	   v_e (x,y,t) = 	\cos t \sin(2\pi x)\cos(2\pi y) . 
\end{aligned} 
\end{equation} 
In turn, two artificial source terms have to be added to the right hand side of the Navier-Stokes equation \eqref{equation-CHNS-1} and the Cahn-Hilliard equation \eqref{equation-CHNS-2}, to make these analytic functions satisfy the PDE system. A sequence of spatial mesh and time step sizes are taken, with $h=2^{-k}$, $k=4$, 5, 6 ,7, 8, 9, and $\tau =h$. The final time is set as $T=1$, and a careful comparison between the exact and the numerical solutions have indicated the full second order convergence rate for the phase variable, velocity and the pressure. More details could be found in~\cite{chen24a}. 
\end{rem}

\begin{rem} 
The linear refinement requirement, $C_1 h \le \tau \le C_2 h$, is imposed in the convergence analysis. In more details, the requirement $\tau \le C_2 h$ is used to balance the inverse inequality, as reported in the rough estimates~\eqref{a priori-2}, \eqref{convergence-rough-11-4}-\eqref{convergence-rough-11-5}, etc. Meanwhile, the requirement $\tau \ge C_1 h$ is needed to recover the a-priori assumption~\eqref{a priori-1}, due to the $\tau^{-1}$ factor for the pressure variable error estimate in~\eqref{a priori-5}. On the other hand, this linear refinement is just a theoretical requirement, and such a constraint is not necessary in the practical numerical implementations. 

In fact, if an even higher order asymptotic expansion is performed in space (corresponding to higher than fourth order accuracy in space), the requirement $\tau \ge C_1 h$ could be reduced to $\tau \ge C_1 h^{\beta_0}$, with $\beta_0 > 1$. Moreover, the value of $\beta_0$ could grow larger, if the spatial asymptotic expansion order becomes larger. Similarly, if a higher than third order asymptotic expansion is performed in time, the requirement $\tau \le C_2 h$ could be reduced to $\tau \le C_2 h^{\alpha_0}$, with $0 < \alpha_0 < 1$, and the value of $\alpha_0$ may become smaller if the temporal asymptotic expansion order becomes larger. Because of these two facts, we conclude that, the convergence analysis would always be valid for any power scaling law between $\tau$ and $h$, as long as $\tau \to 0+$, $h \to 0+$, since an even higher asymptotic expansion would accomplish this analysis. Of course, a theoretical justification of this conclusion will be a tedious process, although the key scientific ideas have been reported in this article. The technical details are left to interested readers.  

Various numerical experiments in an existing work~\cite{chen24a} have also validated this conclusion, in which all the numerical examples, with different ratios of $\tau / h$, have created robust simulation results.   
\end{rem}

\section{Concluding remarks}  \label{sec:conclusion}

In this paper we have presented an optimal rate convergence analysis and error estimate for a second order accurate in time, finite difference numerical scheme for the Cahn-Hilliard-Navier-Stokes system, with logarithmic Flory-Huggins energy potential. A modified Crank-Nicolson approximation is applied to the chemical potential, combined with a nonlinear artificial regularization term. The numerical scheme has been recently proposed, and the positivity-preserving property of the logarithmic arguments, as well as the total energy stability analysis, were justified. In this paper, a second order convergence of the proposed numerical scheme, in both time and space, has been established at a theoretical level. In more details, the $\ell^\infty (0, T; H_h^1) \cap \ell^2 (0, T; H_h^3)$ error estimate for the phase variable and the $\ell^\infty (0, T; \ell^2)$ estimate for the velocity variable, which shares the same regularity as the energy norm, is performed to pass through the nonlinear analysis for the error terms associated with the coupled physical process. In addition, a uniform distance between the numerical solution and the singular limit values has to be derived, so that the nonlinear errors associated with the logarithmic terms could be effectively controlled. To accomplish these goals, a higher order asymptotic expansion of the numerical solution (up to the third order accuracy in time and fourth order in space) has to be performed, so that an application of inverse inequality is able to ensure the phase separation property. A rough error estimate is used to establish the maximum norm bound for the phase variable, so that the phase separation property becomes available for the numerical solution as well. In turn, a more refined nonlinear error bound could be established, which is helpful to the desired convergence result. In the authors' knowledge, this is the first work to establish a second order optimal rate convergence estimate for the Cahn-Hilliard-Navier-Stokes system with a singular energy potential.

	\section*{Acknowledgments} 
This work is supported in part by the grants NSFC 12241101 and NSFC 12071090 (W.~Chen), NSF DMS-2012269,   DMS-2309548 (C.~Wang), NSFC 12271237 and the Havener Fund (X.~Wang). 
C.~Wang also thanks College of Mathematical Sciences, Fudan University,
for the support during his visit.

\appendix

\section{Proof of Proposition~\ref{prop: nonlinear est-1}} 

The following term is denoted for simplicity of presentation: 
\begin{equation} 
  D (\Psi, \psi) = \frac{\ln ( 1+ \Psi) - \ln ( 1+ \psi)}{\Psi - \psi} , \quad \mbox{so that} \, \, \, 
 \ln ( 1+ \Psi) - \ln ( 1+ \psi) = D (\Psi, \psi) \tilde{\psi} .  
 \label{prop 1-1} 
\end{equation} 
The intermediate value theorem indicates that $D (\Psi, \psi) = \frac{1}{1+\xi_1}$, in which $\xi_1$ is between $\psi$ and $\Psi$. Meanwhile, by the separation property~\eqref{separation-2}, we see that $| D (\Psi, \psi) | = \frac{1}{| 1 + \xi_1 |} \le \delta^{-1}$ at a point-wise level. In turn, we arrive at 
\begin{equation} 
  \| D (\Psi, \psi) \|_\infty \le \delta^{-1} ,  \quad \mbox{so that} \, \, \, 
  \|  \ln ( 1+ \Psi) - \ln ( 1+ \psi) \|_2 \le \| D (\Psi, \psi) \|_\infty \cdot \| \tilde{\psi} \|_2 
  \le \delta^{-1} \| \tilde{\psi} \|_2 . 
  \label{prop 1-2}   
\end{equation} 
Using similar arguments, we are able to derive the other inequalities in~\eqref{nonlinear est-0-1}: 
\begin{equation} 
  \|  \ln ( 1- \Psi) - \ln ( 1- \psi) \|_2  \le \delta^{-1} \| \tilde{\psi} \|_2  , \quad 
  \| {\mathcal N} ( \Psi ) - {\mathcal N} ( \psi ) \|_2 \le 2 \delta^{-1} \| \tilde{\psi} \|_2 . 
  \label{prop 1-3}   
\end{equation} 
As a result, \eqref{nonlinear est-0-1} is proved by taking $C_\delta = 2 \delta^{-1}$. 

To obtain a gradient estimate of the nonlinear logarithmic error, we make use of the Taylor expansion for $D (\Psi, \psi)$: 
\begin{equation} 
  D (\Psi, \psi) = \frac{1}{1 + \frac{\Psi + \psi}{2}} + \frac{(\Psi - \psi)^2}{24} \Big( \frac{1}{(1 + \eta_1)^3} + \frac{1}{(1 + \eta_2)^3} \Big) , 
  \label{prop 1-4}   
\end{equation} 
in which $\eta_1$ is between $\Psi$ and $\frac{\Psi + \psi}{2}$, $\eta_2$ is between $\psi$ and $\frac{\Psi + \psi}{2}$, respectively. Meanwhile, because of of the separation property~\eqref{separation-2}, we observe that 
\begin{equation} 
\begin{aligned} 
  & 
  -1 + \delta \le \Psi , \, \psi , \, \frac{\Psi + \psi}{2} , \, \eta_1 , \, \eta_2 \le 1 - \delta ,  
  \quad  \mbox{so that} 
\\
  &
  \Big\| \frac{1}{1 + \frac{\Psi + \psi}{2}} \Big\|_\infty \le \delta^{-1} , \quad 
  \Big\| \frac{1}{(1 + \eta_1)^3} \Big\|_\infty , \, 
  \Big\| \frac{1}{(1 + \eta_2)^3} \Big\|_\infty \le \delta^{-3} , \quad 
  \| D (\Psi, \psi) \|_\infty \le \delta^{-1} + 1 ,  
\end{aligned} 
  \label{prop 1-5}  
\end{equation} 
in which the last inequality comes from the additional condition that $\| \tilde{\psi} \|_\infty = \| \Psi - \psi \|_\infty \le \tau + h$. On the other hand, by taking a finite difference of the expansion of $D (\Psi, \psi)$ in~\eqref{prop 1-4}, a $W_h^{1, \infty}$ bound could be derived: 
\begin{equation} 
\begin{aligned} 
  \| \nabla_h D ( \Psi, \psi) \|_\infty 
  \le & \Big\| \frac{1}{1 + \frac{\Psi + \psi}{2}} \Big\|_\infty^2 
  \cdot \frac12 ( \| \nabla_h \Psi \|_\infty + \| \nabla_h \psi \|_\infty )  
\\
  & 
  + \frac{1}{12} \| \tilde{\psi} \|_\infty \cdot ( \| \nabla_h \Psi \|_\infty + \| \nabla_h \psi \|_\infty )  
  \cdot \Big( \Big\| \frac{1}{(1 + \eta_1)^3} \Big\|_\infty  
  + \Big\| \frac{1}{(1 + \eta_2)^3} \Big\|_\infty \Big) 
\\
  & 
  + \frac{\| \tilde{\psi} \|_\infty^2}{24}  \Big( \Big\| \nabla_h ( \frac{1}{(1 + \eta_1)^3} ) \Big\|_\infty  
  + \Big\| \nabla_h \frac{1}{(1 + \eta_2)^3} ) \Big\|_\infty \Big) 
\\
  \le & 
  \delta^{-2} \cdot \frac12 ( C^* + \tilde{C}_1) 
  + \frac{\tau +h}{12} \cdot ( C^* + \tilde{C}_1)  \cdot 2 \delta^{-3} 
  + \frac{(\tau +h)^2}{24} \cdot \frac{\delta^{-3}}{h} 
\\
  \le & 
   \frac{\delta^{-2}}{2} ( C^* + \tilde{C}_1)  + 1 , 
\end{aligned} 
  \label{prop 1-6}  
\end{equation} 
provided that $\tau$ and $h$ are sufficiently small. Notice that an inverse inequality has been applied at the second step. Therefore, the following gradient estimate becomes available for the logarithmic error term: 
\begin{equation} 
\begin{aligned} 
   & 
    \| \nabla_h ( \ln ( 1 + \Psi) - \ln ( 1 + \psi ) ) \|_2  
    = \| \nabla_h ( D (\Psi, \psi)  \tilde{\psi} ) \|_2 
\\
  \le & 
     \| D (\Psi, \psi)  \|_\infty \cdot \| \nabla_h \tilde{\psi}  \|_2  
     + \| \nabla_h D (\Psi, \psi)  \|_\infty \cdot \| \tilde{\psi}  \|_2 
\\
  \le & 
    (\delta^{-1} + 1) \| \nabla_h \tilde{\psi}  \|_2  
     +  ( \frac{\delta^{-2}}{2} ( C^* + \tilde{C}_1)  + 1 )  \|  \tilde{\psi}  \|_2  . 
\end{aligned} 
  \label{prop 1-7}  
\end{equation} 
The other logarithmic error term could be analyzed in the same manner. Inequality~\eqref{nonlinear est-0-2} is proved, by taking $C_\delta = 2 \max \Big( \delta^{-1} + 1 , \frac{\delta^{-2}}{2} ( C^* + \tilde{C}_1)  + 1 \Big)$. This finishes the proof of Proposition~\ref{prop: nonlinear est-1}.

\section{Proof of Proposition~\ref{prop: nonlinear est-2}} 

The term $F_{1+ \breve{\Phi}^n}(1+ \breve{\Phi}^{n+1}) - F_{1+\phi^n}(1+\phi^{n+1})$ could be decomposed as 
\begin{equation}
	\begin{aligned}
		&F_{1+\breve{\Phi}^n}(1+\breve{\Phi}^{n+1})-F_{1+\phi^n}(1+\phi^{n+1})\\
		& =F_{1+ \breve{\Phi}^{n+1}}(1+ \breve{\Phi}^{n}) - F_{1+ \breve{\Phi}^{n+1}}(1+\phi^n) + F_{1+\phi^n}(1+ \breve{\Phi}^{n+1}) - F_{1+\phi^n}(1+\phi^{n+1})\\
		& = Q ( \breve{\Phi}^{n+1}, \breve{\Phi}^n , \phi^n ) \ephi^n 
		+ Q ( \phi^n , \breve{\Phi}^{n+1} , \phi^{n+1} ) \ephi^{n+1}\\		
	        & =F_{1+\Phi^{n+1}}'(\xi_3)\ephi^n+F_{1+\phi^n}'(\xi_4)\ephi^{n+1} 
		=\frac{\ephi^n}{2 ( 1+ \eta_3)}+\frac{\ephi^{n+1}}{2 (1 + \eta_4)} ,
	\end{aligned} 
	\label{prop 2-1}  
\end{equation}
in which $Q (a, x,y)$ corresponds to the following difference quotient function: 
\begin{equation} 
  F_{1+a} (1+y) - F_{1+a} (1+x) = Q (a, y,x) (y-x) . 
\end{equation} 
Notice that the mean value theorem and Lemma~\ref{estimate of Fa} have been applied in the derivation, $\xi_3$ is between $\phi^n$ and $\breve{\Phi}^n$, $\xi_4$ is between $\phi^{n+1}$ and $\breve{\Phi}^{n+1}$, $\eta_3$ is between $\xi_3$ and $\breve{\Phi}^{n+1}$, $\eta_4$ is between $\xi_4$ and $\phi^{n}$. Because of the phase separation property for $\phi^n$, as well as for $\breve{\Phi}^n$ and $\breve{\Phi}^{n+1}$, we see that 
\begin{equation} 
\begin{aligned} 
  & 
  -1 + \delta \le \eta_3 \le 1 - \delta , \quad \mbox{so that} \, \, \, 
  0 < \frac{1}{2 ( 1+ \eta_3)} \le \frac12 \delta^{-1} , 
\\
  & 
 \langle \ephi^{n+1} , F_{1 + \breve{\Phi}^{n+1}}(1 + \breve{\Phi}^n) - F_{1 + \breve{\Phi}^{n+1}}(1+\phi^n) \rangle_c  
\\
  & 
 =  \langle \ephi^{n+1} , \frac{\ephi^n}{2 ( 1+ \eta_3)}  \rangle_c 
 \ge - \Big\| \frac{1}{2 ( 1+ \eta_3)} \Big\|_\infty \cdot \| \ephi^{n+1} \|_2 \cdot \| \ephi^n \|_2  
 \ge \frac12 \delta^{-1}  \| \ephi^{n+1} \|_2 \cdot \| \ephi^n \|_2 .   
\end{aligned} 
  \label{prop 2-2}  
\end{equation} 
The other nonlinear inner product turns out to be non-negative, due to the positive value of $1 + \eta_4$: 
\begin{equation}  
 \langle \ephi^{n+1} , F_{1 + \phi^n}(1 + \breve{\Phi}^{n+1}) - F_{1 + \phi^n}(1+\phi^{n+1}) \rangle_c   
 =  \langle \ephi^{n+1} , \frac{\ephi^{n+1}}{2 ( 1+ \eta_4)}  \rangle_c 
 \ge 0 .   
  \label{prop 2-3}  
\end{equation} 
Consequently, a combination of~\eqref{prop 2-2} and \eqref{prop 2-3} results in the first inequality of~\eqref{nonlinear est-1-1}, by taking $C_\delta =  \frac12 \delta^{-1}$. The other inequality of~\eqref{nonlinear est-1-1} could be proved in the same manner, and the details are skipped for the sake of brevity. 

In addition, the gradient error estimate could be established in a similar fashion as in the proof of Proposition~\ref{prop: nonlinear est-1}. With the help of the separation property for $\breve{\Phi}^{n+1}$, $\breve{\Phi}^n$, $\phi^{n+1}$ and $\phi^n$, as well the additional conditions~\eqref{nonlinear est-condition-2}, the following $W_h^{1, \infty}$ bounds could be derived for $Q ( \breve{\Phi}^{n+1}, \breve{\Phi}^n , \phi^n )$ and $Q ( \phi^n , \breve{\Phi}^{n+1} , \phi^{n+1} )$:  
\begin{equation} 
\begin{aligned} 
  & 
  \| Q ( \breve{\Phi}^{n+1}, \breve{\Phi}^n , \phi^n ) \|_\infty , \, \,  
  \| Q ( \phi^n , \breve{\Phi}^{n+1} , \phi^{n+1} ) \|_\infty \le \frac12 \delta^{-1} , 
\\
  & 
  \| \nabla_h Q ( \breve{\Phi}^{n+1}, \breve{\Phi}^n , \phi^n ) \|_\infty , \, \,  
  \| \nabla_h Q ( \phi^n , \breve{\Phi}^{n+1} , \phi^{n+1} ) \|_\infty \le M_\delta , \quad 
  \mbox{only dependent on $\delta$ and $C^*$} . 
\end{aligned} 
  \label{prop 2-4}  
\end{equation} 
Again, the technical details are skipped for the sake of brevity. By the decomposition representation~\eqref{prop 2-1}, we see that 
\begin{equation}
	\begin{aligned}
		& 
		\|  \nabla_h ( F_{1+\breve{\Phi}^n}(1+\breve{\Phi}^{n+1}) 
		 - F_{1+\phi^n}(1+\phi^{n+1}) ) \|_2 \\
		\le & 
		  \| \nabla_h ( Q ( \breve{\Phi}^{n+1}, \breve{\Phi}^n , \phi^n ) \ephi^n ) \|_2 
	       + \| \nabla_h ( Q ( \phi^n , \breve{\Phi}^{n+1} , \phi^{n+1} ) \ephi^{n+1} ) \|_2 
		\\ 		
	        \le & 
	        \| \nabla_h Q ( \breve{\Phi}^{n+1}, \breve{\Phi}^n , \phi^n )  \|_\infty 
	        \cdot \| \ephi^n \|_2 
	        + \| Q ( \breve{\Phi}^{n+1}, \breve{\Phi}^n , \phi^n )  \|_\infty 
	        \cdot \| \nabla_h \ephi^n \|_2 
	        \\
	          & 
	          + \| \nabla_h Q ( \phi^n , \breve{\Phi}^{n+1} , \phi^{n+1} ) \|_\infty 
	          \cdot \| \ephi^{n+1}  \|_2 
	          + \| Q ( \phi^n , \breve{\Phi}^{n+1} , \phi^{n+1} ) \|_\infty 
	          \cdot \| \nabla_h \ephi^{n+1}  \|_2 
	        \\
	        \le & 
	        \frac12 \delta^{-1} ( \| \nabla_h \ephi^n \|_2 + \| \nabla_h \ephi^{n+1} \|_2 ) 
	        + M_\delta ( \| \ephi^n \|_2 + \| \ephi^{n+1} \|_2 ) . 
	\end{aligned} 
	\label{prop 2-5}  
\end{equation}
Therefore, the first inequality in \eqref{nonlinear est-1-2} has been proved, by taking $C_\delta = \max( C_0 M_\delta , \frac12 \delta^{-1})$. The other inequality in \eqref{nonlinear est-1-2} could be established in the same style. This finishes the proof of Proposition~\ref{prop: nonlinear est-2}.

	\bibliographystyle{plain}
	\bibliography{refs}

\end{document}